\documentclass[10pt]{amsart}
\usepackage{latexsym,amsmath,amssymb,amscd}
\def\today{\ifcase \month \or
   January \or February \or March \or April \or
   May \or June \or July \or August \or
   September \or October \or November \or December \fi
   \space\number\day , \number\year}

\begin{document}

\DeclareRobustCommand{\SkipTocEntry}[4]{} 

\makeatletter
\@addtoreset{figure}{section}
\def\thefigure{\thesection.\@arabic\c@figure}
\def\fps@figure{h,t}
\@addtoreset{table}{bsection}

\def\thetable{\thesection.\@arabic\c@table}
\def\fps@table{h, t}
\@addtoreset{equation}{section}
\def\theequation{
\arabic{equation}}
\makeatother

\newcommand{\bfi}{\bfseries\itshape}

\newtheorem{theorem}{Theorem}
\newtheorem{acknowledgment}[theorem]{Acknowledgment}
\newtheorem{algorithm}[theorem]{Algorithm}
\newtheorem{axiom}[theorem]{Axiom}
\newtheorem{case}[theorem]{Case}
\newtheorem{claim}[theorem]{Claim}
\newtheorem{conclusion}[theorem]{Conclusion}
\newtheorem{condition}[theorem]{Condition}
\newtheorem{conjecture}[theorem]{Conjecture}
\newtheorem{construction}[theorem]{Construction}
\newtheorem{corollary}[theorem]{Corollary}
\newtheorem{criterion}[theorem]{Criterion}
\newtheorem{definition}[theorem]{Definition}
\newtheorem{example}[theorem]{Example}
\newtheorem{lemma}[theorem]{Lemma}
\newtheorem{notation}[theorem]{Notation}
\newtheorem{problem}[theorem]{Problem}
\newtheorem{proposition}[theorem]{Proposition}
\newtheorem{remark}[theorem]{Remark}
\numberwithin{theorem}{section}
\numberwithin{equation}{section}

\newcommand{\todo}[1]{\vspace{5 mm}\par \noindent
\framebox{\begin{minipage}[c]{0.95 \textwidth}
\tt #1 \end{minipage}}\vspace{5 mm}\par}

\newcommand{\1}{{\bf 1}}
\newcommand{\Ad}{{\rm Ad}}
\newcommand{\Alg}{{\rm Alg}\,}
\newcommand{\Aut}{{\rm Aut}\,}
\newcommand{\ad}{{\rm ad}}
\newcommand{\Cb}{{{\mathcal C}_b}}
\newcommand{\Ci}{{\mathcal C}^\infty}
\newcommand{\de}{{\rm d}}
\newcommand{\ee}{{\rm e}}
\newcommand{\ev}{{\rm ev}}
\newcommand{\fin}{{\rm fin}}
\newcommand{\id}{{\rm id}}
\newcommand{\ie}{{\rm i}}
\newcommand{\GL}{{\rm GL}}
\newcommand{\gl}{{{\mathfrak g}{\mathfrak l}}}
\newcommand{\Hom}{{\rm Hom}\,}
\newcommand{\Img}{{\rm Im}\,}
\newcommand{\Ker}{{\rm Ker}\,}
\newcommand{\lf}{{\rm l}}
\newcommand{\Lie}{\text{\bf L}}
\newcommand{\LUCb}{{{\mathcal L}{\mathcal U}{\mathcal C}_b}}
\newcommand{\m}{\textbf{m}}
\newcommand{\OO}{{\rm O}}
\newcommand{\Part}{{\rm Part}}
\newcommand{\pr}{{\rm pr}}
\newcommand{\Ran}{{\rm Ran}\,}
\newcommand{\rank}{{\rm rank}\,}
\renewcommand{\Re}{{\rm Re}\,}
\newcommand{\RUCb}{{{\mathcal R}{\mathcal U}{\mathcal C}_b}}
\newcommand{\rb}{\text{\bf r}}
\newcommand{\Sp}{{\rm Sp}}
\renewcommand{\sp}{{{\mathfrak s}{\mathfrak p}}}
\newcommand{\spa}{{\rm span}\,}
\renewcommand{\ss}{{\rm ss}}
\newcommand{\Tr}{{\rm Tr}\,}
\newcommand{\UCb}{{{\mathcal U}{\mathcal C}_b}}

\newcommand{\G}{{\rm G}}
\newcommand{\U}{{\rm U}}

\newcommand{\Ac}{{\mathcal A}}
\newcommand{\Bc}{{\mathcal B}}
\newcommand{\Cc}{{\mathcal C}}
\newcommand{\Ec}{{\mathcal E}}
\newcommand{\Dc}{{\mathcal D}}
\newcommand{\Hc}{{\mathcal H}}
\newcommand{\Ic}{{\mathcal I}}
\newcommand{\Jc}{{\mathcal J}}
\newcommand{\Lc}{{\mathcal L}}
\newcommand{\Nc}{{\mathcal N}}
\newcommand{\Oc}{{\mathcal O}}
\newcommand{\Pc}{{\mathcal P}}
\newcommand{\Qc}{{\mathcal Q}}
\newcommand{\Rc}{{\mathcal R}}
\newcommand{\Tc}{{\mathcal T}}
\newcommand{\Vc}{{\mathcal V}}
\newcommand{\Uc}{{\mathcal U}}

\newcommand{\Bg}{{\mathfrak B}}
\newcommand{\Fg}{{\mathfrak F}}
\newcommand{\Gg}{{\mathfrak G}}
\newcommand{\Ig}{{\mathfrak I}}
\newcommand{\Jg}{{\mathfrak J}}
\newcommand{\Lg}{{\mathfrak L}}
\newcommand{\Mg}{{\mathfrak M}}
\newcommand{\Pg}{{\mathfrak P}}
\newcommand{\Sg}{{\mathfrak S}}
\newcommand{\Xg}{{\mathfrak X}}
\newcommand{\Yg}{{\mathfrak Y}}
\newcommand{\Zg}{{\mathfrak Z}}

\newcommand{\ag}{{\mathfrak a}}
\newcommand{\bg}{{\mathfrak b}}
\newcommand{\dg}{{\mathfrak d}}
\renewcommand{\gg}{{\mathfrak g}}
\newcommand{\hg}{{\mathfrak h}}
\newcommand{\kg}{{\mathfrak k}}
\newcommand{\mg}{{\mathfrak m}}
\newcommand{\n}{{\mathfrak n}}
\newcommand{\og}{{\mathfrak o}}
\newcommand{\pg}{{\mathfrak p}}
\newcommand{\sg}{{\mathfrak s}}
\newcommand{\tg}{{\mathfrak t}}
\newcommand{\ug}{{\mathfrak u}}
\newcommand{\zg}{{\mathfrak z}}

\makeatletter
\title[Infinite-dimensional reductive Lie groups]{Functional Analytic Background for 
a Theory of Infinite-Dimensional Reductive Lie Groups}
\thanks{
\texttt{Date: \today}
}
\author{Daniel Belti\c t\u a
}
\address{Institute of Mathematics ``Simion Stoilow'' of 
the Romanian Academy, 
P.O. Box 1-764, RO-014700 Bucharest, Romania} 
\email{Daniel.Beltita@imar.ro}
%
\makeatother


\begin{abstract}
Motivated by the interesting and yet scattered developments 
in representation theory of Banach-Lie groups, 
we discuss several  
functional analytic issues which should underlie the notion of 
infinite-dimensional reductive Lie group: 
norm ideals, triangular integrals, operator factorizations, and amenability. 
\end{abstract}

\maketitle

\section{What a reductive Lie group is supposed to be}

\noindent\textbf{Introduction.}
We approach the problem of finding 
an appropriate infinite-dimensional version 
of reductive Lie group. 
The discussion is motivated by the need to have a reasonably general setting 
where the representation theoretic properties of the classical Lie groups 
associated with the Schatten ideals of Hilbert space operators 
---in the sense of \cite{dlH72}--- 
can be investigated in a systematic way.  
Thus the theory of operator ideals 
(see e.g., \cite{GK69}, \cite{GK70}, \cite{DFWW04}, and \cite{KW07}) 
provides 
the natural background for the present exposition. 

The ideas and methods of representation 
theory of finite-dimensional Lie groups
cannot possibly be extended to the setting of Banach-Lie groups 
in a direct manner. 
Any attempt to do that fails because of 
some phenomena specific to the infinite-dimensional Lie groups:  
there exist Lie algebras that do not arise from Lie groups 
(see e.g., \cite{Ne06} or \cite{Be06}), 
closed subgroups of Lie groups need not be 
Lie groups in the relative topology
(see e.g., \cite{Up85}), 
one does not know how to construct smooth structures on homogeneous spaces 
unless one is able to find a direct complement 
of the Lie subalgebra in the ambient Lie algebra 
(see \cite{Up85}, \cite{BP05}, \cite{Ga06}, or \cite{Be06}), 
there exists no Haar measure on 
topological groups which are not locally compact
(\cite{We40}),  
and finally every infinite-dimensional Banach space 
is the model space of some abelian Lie group 
without any non-trivial continuous representations 
(see \cite{Ba83}). 

Nevertheless, the study of representation theoretic properties 
of some specific Banach-Lie groups has led to a number of 
interesting results; see for instance the papers 
\cite{SV75}, \cite{Bo80}, \cite{Pi90}, \cite{Ne98}, \cite{NO98}, 
\cite{BR06} or \cite{BG07}. 
It seems reasonable to try to find out 
a class of Banach-Lie groups whose representations 
can be studied in a coherent fashion following the pattern 
of representation theory of finite-dimensional reductive Lie groups. 
The aim of the present paper is to survey some of the ideas 
and notions that might eventually lead to the description 
of a class of Banach-Lie groups appropriate 
for the purposes of such a representation theory. 
The exposition is streamlined by a number of phenomena 
that play a central role 
in the classical theory of reductive Lie groups:  
Cartan decompositions, Iwasawa decompositions, Harish-Chandra decompositions, 
and existence of invariant measures. 
By way of describing appropriate versions of these phenomena 
in infinite dimensions,  
the paper provides a self-contained discussion of a number of  
functional analytic issues which should underlay the notion of 
infinite-dimensional reductive Lie group: 
triangular integrals, operator factorizations, and amenability. 

\medbreak

\noindent\textbf{Finite-dimensional reductive Lie groups.}
In order to make clear the structures we shall be looking for 
in infinite dimensions, we now recall the classical setting.  
Remarks \ref{red_alg} and \ref{red_gr} basically concern 
the matrix Lie algebras/groups, where reductivity means 
stability under the operation of taking adjoints of matrices. 
(These remarks will be our main motivation for 
the discussion of $\Phi$-reductivity in Section~\ref{Sect5}.) 
The following general definition is the one used in \cite{Kn96}. 

\begin{definition}\label{reductive}
\normalfont
A finite-dimensional \emph{reductive Lie group} 
is actually a 4-tuple $(G,K,\theta,B)$, 
where $G$ is a finite-dimensional Lie group with 
the Lie algebra $\gg$, 
$K$ is a compact subgroup of $G$ with the Lie algebra $\kg$, 
$\theta\colon\gg\to\gg$ is an involutive automorphism, 
and $B\colon\gg\times\gg\to{\mathbb R}$ is a 
nondegenerate symmetric bilinear form 
which is $\Ad(G)$-invariant and $\theta$-invariant, 
such that the following conditions are satisfied: 
\begin{itemize}
\item[{\rm(i)}] $\gg$ is a reductive Lie algebra with 
the complexification denoted $\gg^{\mathbb C}$; 
\item[{\rm(ii)}] $\kg=\Ker(\theta-\1)$; 
\item[{\rm(iii)}] if we denote $\pg=\Ker(\theta+\1)$, 
then $\gg=\kg\dotplus\pg$ 
and we have $B(\kg,\pg)=\{0\}$, 
$B$ is positive definite on $\pg$, 
and negative definite on $\kg$; 
\item[{\rm(iv)}] the mapping 
$K\times\pg\to G$, $(k,X)\mapsto k\cdot\exp_GX$, 
is a diffeomorphism; 
\item[{\rm(v)}] for every $g\in G$ the automorphism 
$\Ad(g)\colon\gg^{\mathbb C}\to\gg^{\mathbb C}$ 
belongs 
to the connected component 
of $\1\in\Aut(\gg^{\mathbb C})$. 
\end{itemize} 
We say that the Lie group $G$ itself is a reductive Lie group 
if the other items $K$, $\theta$, and $B$ are clear from the context. 
\qed
\end{definition}

\begin{remark}\label{red_alg}
\normalfont
Let $\gg$ be a real finite-dimensional Lie algebra. 
Then $\gg$ is a \emph{reductive Lie algebra} if and only if 
it is isomorphic to a real Lie subalgebra $\gg_1$ 
of the matrix algebra $M_n({\mathbb C})$ for some integer $n\ge1$
such that for every $X\in\gg_1$ we have $X^*\in\gg_1$. 
In addition, if $\gg$ were a complex Lie algebra, 
then $\gg_1$ could be chosen to be 
a complex Lie subalgebra of $M_n({\mathbb C})$. 
\qed
\end{remark}

\begin{remark}\label{red_gr}
\normalfont
Let $G$ be a connected closed subgroup of $\GL(n,{\mathbb C})$ 
for some $n\ge1$ such that for every $g\in G$ we have $g^*\in G$. 
Then $G$ is a reductive  Lie group with $K=G\cap\U(n)$, 
$\theta(X)=-X^*$ for all $X\in\gg=\Lie(G)\subseteq M_n({\mathbb C})$ 
and $B(X,Y)=\Re(\Tr(XY))$ for all $X,Y\in\gg$ 
(see Example~2 in Section~2 of Chapter~VII of \cite{Kn96}). 

It is not difficult to prove that coverings with finitely many sheets 
of any group $G$ as above are reductive Lie groups 
---the covering groups of this type are precisely 
the \emph{connected} reductive groups in the sense of Definition~2.5 
of \cite{Vo00}. 
(See Proposition~\ref{LL2} below for the way the Cartan decompositions 
lift to covering groups.) 
In particular every connected semisimple Lie group with finite center 
is a reductive Lie group in the sense of Definition~\ref{reductive} above 
(Example~1 in Section~2 of Chapter~VII of \cite{Kn96}).
\qed
\end{remark}

\noindent\textbf{Some conventions and notation.}
Throughout the present paper we denote by 
$\Hc$ an infinite-dimen\-sion\-al complex separable Hilbert space, 
by $\Bc(\Hc)$ the set of all bounded linear operators on $\Hc$ 
and by $\Fg$ the two-sided ideal of $\Bc(\Hc)$ 
consisting of all finite-rank operators. 
Some convenient references for infinite-dimensional Lie groups 
with a functional analytic flavor are 
\cite{dlH72}, \cite{Up85}, \cite{Ne04}, \cite{Ne06}, and \cite{Be06}. 
As in the latter reference, 
we shall always denote by $\Lie(\cdot)$ the Lie functor 
from the category of Banach-Lie groups into 
the category of Banach-Lie algebras.  
We also adopt the convention that Lie groups are denoted by 
Roman capitals and their Lie algebras are denoted by 
the corresponding Gothic lower case letters.

\section{Triangular integrals and factorizations} 

\noindent\textbf{Norm ideals.}
The norm ideals will play a critical role 
for the present exposition.  
In fact our candidates for infinite-dimensional reductive groups 
will be Banach-Lie groups whose model spaces are norm ideals.  

\begin{definition}\label{ideals1}
\normalfont
By \emph{norm ideal}
we mean a two-sided ideal
$\mathfrak{I}$ of
${\mathcal B}({\mathcal H})$
equipped with a norm $\|\cdot\|_\mathfrak{I}$ satisfying
$\|T\|\le\|T\|_\mathfrak{I}=\|T^*\|_\mathfrak{I}$ and
$\|ATB\|_\mathfrak{I}\le\|A\|\,\|T\|_\mathfrak{I} \,\|B\|$
whenever $A,B\in{\mathcal B}({\mathcal H})$.
\qed
\end{definition}

\begin{definition}\label{ideals2}
\normalfont
Let $\widehat{c}$ be the vector space of all sequences of
real
numbers $\{\xi_j\}_{j\ge1}$
such that $\xi_j=0$ for all but finitely many indices.
A \emph{symmetric norming function}
is a function
$\Phi\colon\widehat{c}\to{\mathbb R}$
satisfying the following conditions:
\begin{itemize}
\item[{\rm(i)}] the function $\Phi(\cdot)$ is a norm on 
the linear space $\widehat{c}$ and $\Phi((1,0,0,\dots))=1$; 
\item[{\rm(ii)}] $\Phi(\{\xi_j\}_{j\ge1})=\Phi(\{\xi_{\pi(j)}\}_{j\ge1})$
whenever $\{\xi_j\}_{j\ge1}\in\widehat{c}$ and
$\pi\colon{\mathbb N}\setminus\{0\}\to{\mathbb N}\setminus\{0\}$
is bijective.
\end{itemize}
Any symmetric norming function $\Phi$ gives rise to two norm ideals
${\mathfrak S}_\Phi$ and ${\mathfrak S}_\Phi^{(0)}$
as follows.
For every bounded sequence of real numbers $\xi=\{\xi_j\}_{j\ge1}$ define
$\Phi(\xi):=\sup_{n\ge1}\Phi(\xi_1,\xi_2,\dots,\xi_n,0,0,\dots)\in[0,\infty]$.
For all $T\in{\mathcal B}({\mathcal H})$ denote
$$\|T\|_\Phi:=\Phi(\{s_j(T)\}_{j\ge1})\in[0,\infty],$$
where
$s_j(T)=\inf\{\|T-F\|\mid F\in{\mathcal B}({\mathcal H}),\,{\rm
rank}\,F<j\}$ whenever
$j\ge1$.
With this notation we can define
$$
\Sg_\Phi^{(0)}:=\overline{\Fg}^{\|\cdot\|_\Phi}
\subseteq\{T\in\Bc(\Hc)\mid\Vert T\Vert_\Phi<\infty\}=:\Sg_\Phi$$
that is, ${\mathfrak S}_\Phi^{(0)}$ is the $\|\cdot\|_\Phi$-closure of
the ideal of finite-rank operators ${\mathfrak F}$
in ${\mathfrak S}_\Phi$.
Then $\|\cdot\|_\Phi$ is a norm making ${\mathfrak S}_\Phi$ and
${\mathfrak S}_\Phi^{(0)}$
into norm ideals
(see \S 4 in Chapter~III in \cite{GK69}). 
We say that $\Phi$ is a \emph{mononormalizing} symmetric norming function 
if $\Sg_\Phi^{(0)}=\Sg_\Phi$. 
If $1\le p<\infty$, then the formula 
$\Phi_p(\xi)=\|\xi\|_{\ell^p}$ whenever
$\xi\in\widehat{c}$ defines a mononormalizing symmetric norming function 
and the corresponding norm ideal $\Sg_p:=\Sg_{\Phi_p}$ is 
the \emph{$p$-th Schatten ideal}. 
In the special case $p=2$ we call
$\Sg_2$ the \emph{Hilbert-Schmidt ideal}.
\qed
\end{definition}

\begin{remark}\label{ideals3}
\normalfont
Every separable norm ideal is equal to $\Sg_\Phi^{(0)}$
for some symmetric norming function $\Phi$
(see Theorem~6.2 in Chapter~III in \cite{GK69}).
\qed
\end{remark}

\begin{remark}\label{ideals4}
\normalfont
For every symmetric norming function $\Phi\colon\widehat{c}\to{\mathbb R}$
there exists a unique symmetric norming function
$\Phi^*\colon\widehat{c}\to{\mathbb R}$
such that
$$\Phi^*(\eta)=\sup\Bigl\{
\frac{1}{\Phi(\xi)}\sum\limits_{j=1}^\infty\xi_j\eta_j
\;\Big|\;\xi=\{\xi_j\}_{j\ge1}\in\widehat{c}\text{ and }
\xi_1\ge\xi_2\ge\cdots\ge0
\Bigr\}$$
whenever $\eta=\{\eta_j\}_{j\ge1}\in\widehat{c}$ and
$\eta_1\ge\eta_2\ge\cdots\ge0$.
The function $\Phi^*$ is said to be \emph{adjoint} to $\Phi$
and we always have $(\Phi^*)^*=\Phi$
(see Theorem~11.1 in Chapter~III in \cite{GK69}).
For instance, if $1\le p,q\le\infty$, $1/p+1/q=1$,
$\Phi_p(\xi)=\|\xi\|_{\ell^p}$ and $\Phi_q(\xi)=\|\xi\|_{\ell^q}$ whenever
$\xi\in\widehat{c}$,
then $(\Phi_p)^*=\Phi_q$.
If $\Phi$ is any symmetric norming function then the topological dual of
the Banach space
$\Sg_\Phi^{(0)}$ is isometrically isomorphic to $\Sg_{\Phi^*}$
by the duality pairing
$\Sg_{\Phi^*}\times\Sg_\Phi^{(0)}\to{\mathbb C}$, 
$(T,S)\mapsto\Tr(TS)$
(see Theorems 12.2~and~12.4 in Chapter~III in \cite{GK69}). 
\qed
\end{remark}

The next definition describes the Boyd indices of 
a symmetric norming function. 
Some convenient references for this notion are 
Section~3 in \cite{Ar78} and subsections~2.17--19 in \cite{DFWW04}. 

\begin{definition}\label{boyd}
\normalfont
Let $\Phi$ be a symmetric norming function. 
For each $m\ge1$ define the linear operators 
$D_m\colon\widehat{c}\to\widehat{c}$ and 
$D_{1/m}\colon\widehat{c}\to\widehat{c}$ 
by 
$$ 
D_m\xi=(\underbrace{\xi_1,\dots,\xi_1}_{m\text{ times}},
\underbrace{\xi_2,\dots,\xi_2}_{m\text{ times}},\dots) 
\text{ and }
D_{1/m}\xi=
\Bigl(\frac{1}{m}\sum_{i=1}^m\xi_i,\frac{1}{m}\sum_{i=m+1}^{2m}\xi_i,
\dots\Bigr)
$$
for arbitrary $\xi=(\xi_1,\xi_2,\dots)\in\widehat{c}$. 
We shall think of $D_m$ and $D_{1/m}$ as linear operators 
on the space $\widehat{c}$ equipped with the norm $\Phi(\cdot)$, 
so that it makes sense to speak about the norms of these operators. 
The \emph{Boyd indices} of the symmetric norming function $\Phi$ 
are defined by 
$$p_\Phi=\sup_{m\ge1}\frac{\log m}{\log\Vert D_m\Vert} 
\quad\text{ and }\quad
q_\Phi=\inf_{m\ge1}\frac{\log (1/m)}{\log\Vert D_{1/m}\Vert}$$
and we have $1\le p_\Phi\le q_\Phi\le\infty$. 
We shall say that these indices are \emph{nontrivial}
if $1<p_\Phi\le q_\Phi<\infty$. 
\qed
\end{definition}

\begin{remark}\label{compute_boyd}
\normalfont
The Boyd indices of any symmetric norming function $\Phi$ are related to the 
Boyd indices of its adjoint $\Phi^*$ by the equations 
$1/p_\Phi+1/q_{\Phi^*}=1/p_{\Phi^*}+1/q_\Phi=1$. 
If $1<r<\infty$ and $\Phi(\cdot)=\Vert\cdot\Vert_{\ell^r}$ 
then for all $m\ge1$ we have 
$\Vert D_m\Vert=m^{1/r}$ and $\Vert D_{1/m}\Vert=m^{-1/r}$, 
hence in this case $p_\Phi=q_\Phi=r$. 
\qed
\end{remark}

The Boyd indices of symmetric norming functions are 
important for our present purposes because of 
the following interpolation theorem. 

\begin{theorem}\label{interpolation}
Let $\Phi$ be a mononormalizing symmetric norming function and assume that 
$1\le p<p_\Phi\le q_\Phi<q\le\infty$. 
Then the following assertions hold: 
\begin{itemize}
\item[{\rm(a)}] We have $\Sg_p\subseteq\Sg_\Phi\subseteq\Sg_q$.
\item[{\rm(b)}] There exists a constant $M_\Phi>0$ with 
the following property: 
If $T\colon\Sg_q\to\Sg_q$ is a bounded linear operator such that 
$T(\Sg_p)\subseteq\Sg_p$, then $T(\Sg_\Phi)\subseteq\Sg_\Phi$ 
and 
$\Vert T|_{\Sg_\Phi}\Vert\le 
M_\Phi\max\{\Vert T|_{\Sg_p}\Vert,\Vert T|_{\Sg_q}\Vert\}$.
\end{itemize} 
\end{theorem}

\begin{proof}
See Corollary~3.4(i) in \cite{Ar78}. 
\end{proof}

\noindent\textbf{Triangular integrals.}
The triangular integral is a suitable infinite-dimension\-al version 
of the operation of taking the upper triangular part 
(the lower diagonal part, or the diagonal, respectively) 
of a square matrix. 
The classical reference is \cite{GK70} 
(see also \cite{Er78}). 
We are going to describe this idea 
in a setting which is slightly more general than the classical one. 

\begin{definition}\label{tri1}
\normalfont
Let $\Bg$ be a unital associative involutive Banach algebra 
and $\Ic$ a 
\emph{contractive $\Bg$-bimodule}, that is, 
$\Ic$ is a Banach space equipped with a trilinear map 
$$\Bg\times\Ic\times\Bg\to\Ic,\quad (b_1,X,b_2)\mapsto b_1Xb_2,$$
and an involutive isometric antilinear map 
$\Ic\to\Ic$, $X\mapsto X^*$, 
such that 
$\Ic$ is an algebraic $\Bg$-bimodule and in addition 
$(b_1Xb_2)^*=b_2^*X^*b_1^*$ and  
$\Vert b_1Xb_2\Vert\le\Vert b_1\Vert \Vert X\Vert \Vert b_2\Vert$ 
whenever $b_1,b_2\in\Bg$ and $X\in\Ic$. 

Also let $\Ec$ be a totally ordered set of self-adjoint idempotent 
elements in $\Bg$ such that $0,\1\in\Ec$, 
and denote by $\Part(\Ec)$ the set of all \emph{partitions} 
of $\Ec$, that is, the finite families $\Pc=\{p_i\}_{0\le i\le n}$ 
in $\Ec$ such that $0=p_0<p_1<\cdots<p_n=\1$. 
For such a partition $\Pc$ of $\Ec$ we define the 
\emph{diagonal truncation} operator 
$$\Dc_{\Pc}\colon\Ic\to\Ic,\quad 
\Dc_{\Pc}(X)=\sum_{i=1}^n (p_i-p_{i-1})X(p_i-p_{i-1}), 
$$
the \emph{strictly upper triangular truncation} operator 
$$\Uc_{\Pc}\colon\Ic\to\Ic,\quad 
\Uc_{\Pc}(X)=\sum_{i=1}^n p_{i-1}X(p_i-p_{i-1}), 
$$
and the 
\emph{strictly lower triangular truncation} operator 
$$\Lc_{\Pc}\colon\Ic\to\Ic,\quad 
\Lc_{\Pc}(X)=\sum_{i=1}^n (p_i-p_{i-1})Xp_{i-1}.  
$$
Now think of $\Part(\Ec)$ as a directed set with respect to 
the partial order defined by $\Pc\le\Qc$ 
if and only if $\Pc$ is a subfamily of $\Qc$. 
If $X\in\Ic$ and the net $\{\Uc_{\Pc}(X)\}_{\Pc\in\Part(\Ec)}$ 
is convergent in $\Ic$, 
then the corresponding limit is denoted by $\Uc(X)$ and is called 
the \emph{strictly upper triangular integral} of $X$. 
One similarly defines 
the \emph{strictly lower triangular integral} $\Lc(X)$ 
and the \emph{diagonal integral} $\Dc(X)$ 
whenever they exist. 
\qed
\end{definition}

\begin{proposition}\label{tri2}
Let $\Bg$ be a unital associative involutive Banach algebra, 
$\Ic$ a contractive $\Bg$-bimodule and $\Ec$ a totally ordered 
set of self-adjoint idempotents in $\Bg$ such that $0,1\in\Ec$. 
Assume that the following conditions are satisfied: 
\begin{itemize}
\item[{\rm(i)}] 
Either the Banach space underlying $\Ic$ is reflexive 
or the integrals $\Dc$ and $\Uc$ are convergent on dense subsets of $\Ic$. 
\item[{\rm(ii)}] 
Both families of operators 
$\{\Dc_{\Pc}\}_{\Pc\in\Part(\Ec)}$ and 
$\{\Uc_{\Pc}\}_{\Pc\in\Part(\Ec)}$ 
are uniformly bounded.
\end{itemize} 
Then the following assertions hold: 
\begin{itemize}
\item[{\rm(a)}] 
The strictly upper triangular integral, 
the strictly lower triangular integral, 
and the diagonal integral 
exist throughout $\Ic$ and the corresponding mappings are 
bounded linear idempotent operators 
$\Uc,\Lc,\Dc\colon\Ic\to\Ic$. 
\item[{\rm(b)}] 
There exists the direct sum decomposition 
$$\Ic=\Ran\Lc\dotplus\Ran\Dc\dotplus\Ran\Uc, $$
and the corresponding decomposition of an arbitrary element 
$X\in\Ic$ is $X=\Lc(X)+\Dc(X)+\Uc(X)$. 
In addition, 
$\Lc(X^*)=\Uc(X)^*$ and $\Dc(X^*)=\Dc(X)^*$ for all $X\in\Ic$.  
\end{itemize} 
\end{proposition}

\begin{proof} 
If the integrals $\Dc$ and $\Uc$ are convergent on dense subsets of $\Ic$, 
then they are convergent everywhere because of hypothesis~(ii). 
Now let us assume that the Banach space underlying $\Ic$ is reflexive. 
It is easy to check that for all $\Pc\in\Part(\Ec)$ all of the operators 
$\Lc_{\Pc}$, $\Dc_{\Pc}$, and $\Uc_{\Pc}$ are idempotent,  
their mutual products are equal to $0$, and their sum is the identity 
mapping on $\Ic$. 
In addition, the nets $\{\Uc_{\Pc}(X)\}_{\Pc\in\Part(\Ec)}$  and 
$\{\Lc_{\Pc}(X)\}_{\Pc\in\Part(\Ec)}$ 
are increasing, 
while $\{\Dc_{\Pc}(X)\}_{\Pc\in\Part(\Ec)}$ is decreasing, 
in the sense that if $\Pc,\Qc\in\Part(\Ec)$ and $\Pc\le\Qc$, 
then 
$\Uc_{\Pc}\Uc_{\Qc}=\Uc_{\Qc}\Uc_{\Pc}=\Uc_{\Pc}$, 
$\Lc_{\Pc}\Lc_{\Qc}=\Lc_{\Qc}\Lc_{\Pc}=\Lc_{\Pc}$, 
and 
$\Dc_{\Pc}\Dc_{\Qc}=\Dc_{\Qc}\Dc_{\Pc}=\Dc_{\Qc}$. 
Now hypotheses (i)~and~(ii) ensure that 
Theorem~2.1 of \cite{Er78} applies, 
whence assertion~(a) follows. 

For assertion~(b) note that 
$X=\Lc_{\Pc}(X)+\Dc_{\Pc}(X)+\Uc_{\Pc}(X)$ 
and then take the limit with respect to $\Pc\in\Part(\Ec)$, 
for each $X\in\Ic$. 
For all $\Pc,\Qc\in\Part(\Ec)$ we have 
$\Lc_{\Pc}\Dc_{\Pc}=0$, whence by taking the limit 
we get 
$\Lc\Dc=0$. 
Similarly, all of the mutual products of the operators 
$\Lc$, $\Dc$, and $\Uc$, are equal to $0$, 
whence the asserted direct sum decomposition follows. 
To complete the proof note that for all $X\in\Ic$ and $\Pc\in\Part(\Ec)$ 
we have $\Lc_{\Pc}(X^*)=\Uc_{\Pc}(X)^*$ 
and $\Dc_{\Pc}(X^*)=\Dc_{\Pc}(X)^*$, 
and then again take the limit with respect to the partition $\Pc$ 
of $\Ec$. 
\end{proof}

\begin{example}\label{tri3}
\normalfont
(See Theorem~4.1 in \cite{Ar78}.)
Let $\Phi$ be a mononormalizing symmetric norming function whose 
Boyd indices are nontrivial 
and denote $\Ic=\Sg_\Phi\subseteq\Bc(\Hc)$. 
Denote by $\Ec$ the set of orthogonal projections associated 
with a maximal, totally ordered set of closed subspaces of $\Hc$. 
Then the corresponding strictly upper triangular, 
the strictly lower triangular, and the diagonal integral 
exist throughout $\Ic$ and they define bounded linear idempotent operators 
$\Uc,\Lc,\Dc\colon\Ic\to\Ic$ whose mutual products are equal to $0$ 
and $\Uc+\Lc+\Dc=\1$. 

Indeed, condition~(i) in Proposition~\ref{tri2} 
is satisfied since the integrals of finite-rank operators are 
clearly convergent. 
Also, it follows by Theorem~3.2 in \cite{Er78} that 
the families of operators 
$\{\Dc_{\Pc}\}_{\Pc\in\Part(\Ec)}$ and 
$\{\Uc_{\Pc}\}_{\Pc\in\Part(\Ec)}$ 
are uniformly bounded on each Schatten ideal $\Sg_p$ 
if $1<p<\infty$. 
Now Theorem~\ref{interpolation} shows that both these families 
are uniformly bounded on $\Ic=\Sg_\Phi$ as well, 
since the Boyd indices of $\Phi$ are non-trivial. 
Thus condition~(ii) in Proposition~\ref{tri2} 
is satisfied as well, and it then follows that the integrals 
$\Uc$, $\Dc$, and $\Lc$ are convergent throughout~$\Ic$. 
\qed
\end{example}

\noindent\textbf{Triangular factorizations.}
Our next purpose is to survey some of the methods that 
allow one to find operator theoretic versions 
of the well-known LU~factorization from linear algebra, 
that is, the fact that every invertible matrix factorizes as  
the product of a unitary matrix and a triangular one. 

\begin{definition}\label{nest}
\normalfont
Let $\Ec$ be the set of orthogonal projections associated 
with a maximal, totally ordered set of closed linear subspaces of $\Hc$. 
Then the \emph{nest algebra} of $\Ec$ is 
$\Alg\Ec=\{b\in\Bc(\Hc)\mid(\forall e\in\Ec)\quad be=ebe\}$, 
that is, the set of all operators which leave invariant 
each subspace in the family $\{e(\Hc)\}_{e\in\Ec}$.
\qed
\end{definition}

\begin{definition}\label{pairs}
\normalfont
Let $\Sg_{\rm I}$ and $\Sg_{\rm II}$ be norm ideals. 
We shall say that $(\Sg_{\rm I},\Sg_{\rm II})$ is 
a pair of \emph{associated} norm ideals if 
the following conditions are satisfied: 
\begin{itemize}
\item[{\rm(i)}] $\Sg_{\rm I}$ is separable; 
\item[{\rm(ii)}] $\Sg_{\rm I}\subseteq\Sg_{\rm II}$; 
\item[{\rm(iii)}] for every maximal, totally ordered set of 
closed linear subspaces of $\Hc$ and every $X\in\Sg_{\rm I}$, 
the corresponding strictly upper triangular integral is convergent 
in the contractive $\Bc(\Hc)$-bimodule $\Ic=\Bc(\Hc)$ 
and $\Uc(X)\in\Sg_{\rm II}$. 
\end{itemize}
(See \cite{GK70} and \cite{Er72} for the original version of this definition.)
\qed
\end{definition}

\begin{example}\label{self-assoc}
\normalfont
Let $\Phi$ be a mononormlizing symmetric norming function 
whose Boyd indices are nontrivial, 
and denote $\Sg_{\rm I}=\Sg_{\rm II}=\Sg_\Phi$. 
Then Example~\ref{tri3} shows that $(\Sg_{\rm I},\Sg_{\rm II})$ 
is a pair of associated norm ideals. 

As a special case, it follows by Remark~\ref{compute_boyd} 
that each pair $(\Sg_p,\Sg_p)$ consisting of the $p$-th Schatten ideal and itself 
is a pair of associated norm ideals if $1<p<\infty$. 
\qed
\end{example}

\begin{theorem}\label{erdos}
Let $(\Sg_{\rm I},\Sg_{\rm II})$ be a pair of associated norm ideals 
and consider the triangular and diagonal integrals on 
the contractive $\Bc(\Hc)$-bimodule $\Ic=\Bc(\Hc)$ 
with respect to the set $\Ec$ of orthogonal projections 
associated with some maximal, totally ordered set of 
closed linear subspaces in $\Hc$. 
Also assume that $0\le a\in\GL(\Hc)$. 
Then the following conditions are equivalent: 
\begin{itemize}
\item[{\rm(i)}] $\Uc(a^{-1})$ exists and $\Uc(a^{-1})\in\Sg_{\rm I}$. 
\item[{\rm(ii)}] $\Dc(a^{-1})$ exists and $a^{-1}-\Dc(a^{-1})\in\Sg_{\rm I}$. 
\end{itemize}
If one of these conditions holds true, 
then there exist uniquely determined 
$r\in\Sg_{\rm II}\cap\Alg\Ec$ 
and $d=d^*\in\Alg\Ec$ such that $a=(\1+r)d(\1+r^*)$ and 
the spectrum of $r$ is equal to $\{0\}$. 
\end{theorem}

\begin{proof}
See Theorem~4.2 and Lemma~2.5(i) in \cite{Er72}. 
\end{proof}

Here are two corollaries concerning the group $\GL_\Phi(\Hc)$ 
of Example~\ref{classical} below. 

\begin{corollary}\label{fact1}
Let $\Phi$ be a mononormalizing symmetric norming function 
whose Boyd indices are nontrivial, 
and denote by $\Ec$ the set of orthogonal projections associated 
with a maximal, totally ordered set of closed linear subspaces of $\Hc$. 
If $0\le a\in\GL_\Phi(\Hc)$, then there exist uniquely determined 
operators $d,r\in\Bc(\Hc)$ such that 
$0\le d\in\GL_\Phi(\Hc)\cap\Alg\Ec$, 
$r\in\Sg_\Phi\cap\Alg\Ec$, 
the spectrum of $r$ is equal to $\{0\}$, 
and $a=(\1+r^*)d(\1+r)$.
\end{corollary}

\begin{proof}
See Theorem~\ref{erdos} and Example~\ref{self-assoc}. 
\end{proof}

\begin{corollary}\label{fact2}
Assume the setting of \emph{Corollary~\ref{fact1}}. 
Then for every $g\in\GL_\Phi(\Hc)$ there exist the operators 
$b,u\in\GL_\Phi(\Hc)$ such that $b\in\Alg\Ec$, $u^*u=\1$, 
and $g=ub$. 
\end{corollary}

\begin{proof} 
Just apply Corollary~\ref{fact1} for $a=g^*g$. 
See for instance the proof of Corollary~A.2 in \cite{Be07} for more details. 
\end{proof}

\section{Invariant means on groups}

\noindent\textbf{Amenable groups.}
We shall briefly discuss the invariant means on topological groups. 
These can be thought of as weak versions of Haar measures  
although they have two main drawbacks: 
they may not be faithful, in the sense that the mean of a non-zero function 
with nonnegative values can be equal to zero; 
and not every Lie group admits an invariant mean. 
On the other hand, we shall see that many Banach-Lie groups 
do have invariant means; see for instance Remark~\ref{am}. 
Classical references for amenability 
are \cite{Pa88} and \cite{Ey72}. 
See \cite{Pe06}, \cite{Ga06}, and \cite{BP05} for some recent developments.

\begin{definition}\label{me1}
\normalfont
Let $G$ be a topological group. 
Consider the commutative unital $C^*$-algebra 
$\ell^\infty(G)=\{\psi\colon G\to{\mathbb C}\mid\Vert
 \psi\Vert_\infty:=\sup_G\vert\psi(\cdot)\vert<\infty\} $
and its automorphisms
$L_x,R_x\colon\ell^\infty(G)\to\ell^\infty(G)$, 
$(L_x\psi)(y)=\psi(xy)$ and $(R_x\psi)(y)=\psi(yx)$ 
whenever $y\in G$ and $\psi\in\ell^\infty(G)$,  
defined for arbitrary $x\in G$. 
The space of \emph{right uniformly continuous} bounded functions 
on $G$ is 
$$\RUCb(G)=\{\psi\in\ell^\infty(G)\mid \text{the map }
G\to\ell^\infty(G),\, x\mapsto L_x\psi,\text{ is continuous}\}.  
$$ 
Similarly, the space of 
\emph{left uniformly continuous} bounded functions 
on $G$ is the set $\LUCb(G)$ of all functions 
$\psi\in\ell^\infty(G)$ such that 
the mapping $G\to\ell^\infty(G)$, $x\mapsto R_x\psi$, is continuous. 
And the space of 
\emph{uniformly continuous} bounded functions 
on $G$ is $\UCb(G):=\RUCb(G)\cap\LUCb(G)$. 

We say that the topological group $G$ is \emph{amenable}
if there exists a linear functional 
$\mu\colon\RUCb(G)\to{\mathbb C}$ such that $\mu(\1)=1$, 
$0\le\mu(\psi)$ if $0\le\psi\in\RUCb(G)$, 
and $\mu(L_x\psi)=\mu(\psi)$ 
for all $\psi\in\RUCb(G)$ and $x\in G$. 
In this case we say that the linear functional~$\mu$ 
is a \emph{left invariant mean} on $G$.  
\qed
\end{definition}

\begin{remark}\label{me2}
\normalfont
If the topological group $G$ is amenable, 
then every left invariant mean~$\mu$ 
is continuous on $\RUCb(G)$ and $\Vert\mu\Vert=1$. 
On the other hand, 
the space $\RUCb(G)$ is a unital $C^*$-subalgebra of 
$\ell^\infty(G)$ which consists only of continuous functions, 
and it is invariant under the automorphisms $L_x$ 
for all $x\in G$. 
Thus the left invariant means on $G$ are precisely 
the states of the commutative unital $C^*$-algebra $\RUCb(G)$ 
which are invariant under the automorphism group defined 
by the mappings $L_x$ for arbitrary $x\in G$. 
\qed
\end{remark}

\begin{example}\label{me3}
\normalfont
\emph{If the topological group $G$ is either compact or abelian, 
then it is amenable.} 
In fact, let $\Cc(G)$ denote the space of all continuous functions 
on $G$. If $G$ is compact then the probability Haar measure 
defines a linear functional $\mu\colon\Cc(G)\to{\mathbb C}$. 
Compactness of $G$ implies that $\RUCb(G)=\Cc(G)$, 
and then the basic properties of the Haar measure show that 
$\mu$ is a left invariant mean on $G$. 
On the other hand, if $G$ is abelian, denote by $G_d$ 
the group $G$ endowed with the discrete topology. 
Then $\RUCb(G_d)=\ell^\infty(G)$ and it 
follows by (0.15) in \cite{Pa88} that there exists 
a left invariant mean $\mu\colon\ell^\infty(G_d)\to{\mathbb C}$ 
on the discrete group $G_d$. 
Now the restriction of $\mu$ to $\RUCb(G)$ defines a left invariant mean 
on~$G$. 
\qed
\end{example}

\begin{example}\label{me4}
\normalfont
\emph{Assume that $G$ is a topological group such that 
there exists a directed system of amenable topological subgroups 
$\{G_\alpha\}_{\alpha\in A}$ whose union is dense in $G$. 
Then $G$ is amenable.}
To see this, we shall say that a linear functional 
$\mu\colon\RUCb(G)\to{\mathbb C}$ is a \emph{mean} on $G$ if 
$\mu(\1)=1$ and  
$0\le\mu(\psi)$ whenever $0\le\psi\in\RUCb(G)$. 
In this case $\mu$ is continuous and $\Vert\mu\Vert=1$. 
Now for every $\alpha\in A$ denote 
$$\Lambda_\alpha=\{\mu\mid \mu\text{ is a mean on }G
\text{ and }\mu\circ L_x|_{\RUCb(G)}=\mu\text{ if }
x\in G_\alpha\}, 
$$
which is a $w^*$-compact subset of the unit ball in 
the topological dual space $(\RUCb(G))^*$. 
In addition, $\Lambda_\alpha\ne\emptyset$. 
In fact, any left invariant mean $\mu_\alpha$ on $G_\alpha$ 
gives rise to an element $\widetilde{\mu}_\alpha\in\Lambda_\alpha$ 
defined by 
$\widetilde{\mu}_\alpha(\psi)=\mu_\alpha(\psi|_{G_\alpha})$ 
for all $\psi\in\RUCb(G)$. 
On the other hand, if $\alpha,\beta\in A$ and $G_\alpha\subseteq G_\beta$, 
then $\Lambda_\alpha\supseteq\Lambda_\beta$. 
Thus $\{\Lambda_\alpha\}_{\alpha\in A}$ is a family of 
$w^*$-compact subsets of the unit ball in $(\RUCb(G))^*$ 
with the property that the intersection of each finite subfamily 
is nonempty. 
Therefore $\bigcap\limits_{\alpha\in A}\Lambda_\alpha\ne\emptyset$, 
and  any element $\mu$ in this nonempty intersection  
is a left invariant mean on $G$. 
In fact, we already know that $\mu$ is a mean on $G$. 
To check that it is left invariant, 
let $x\in G$ arbitrary. 
Since the union of the family $\{G_\alpha\}_{\alpha\in A}$ 
is dense in $G$ there exists a net $\{x_i\}_{i\in I}$ 
in that union such that $\lim\limits_{i\in I}x_i=x$. 
Then for every $\psi\in\RUCb(G)$ we have 
$\lim\limits_{i\in I}L_{x_i}\psi=L_x\psi$ in $\ell^\infty(G)$, 
hence $\mu(L_x\psi)=\lim\limits_{i\in I}\mu(L_{x_i}\psi)
=\lim\limits_{i\in I}\mu(\psi)=\mu(\psi)$, 
where the second equality follows since 
$\mu\in\bigcap\limits_{\alpha\in A}\Lambda_\alpha$  and 
$x_i\in\bigcup\limits_{\alpha\in A}G_\alpha$ 
for all $i\in I$. 
\qed
\end{example}

\begin{example}\label{me5}
\normalfont
\emph{Let $G$ be a finite-dimensional Lie group 
with finitely many connected components. 
Denote by $R$ the radical of $G$ 
(i.e., the connected subgroup of $G$ corresponding to the largest 
solvable ideal of the Lie algebra of $G$) 
and by $G_d$ the group $G$ endowed with the discrete topology. 
\begin{itemize}
\item[{\rm(a)}] The group $G$ is amenable if and only if 
                it has any of these properties: 
 \begin{itemize}
 \item[{\rm(i)}] the group $G/R$ is compact; 
 \item[{\rm(ii)}] there exists no closed subgroup of $G$ isomorphic to 
                  the free group ${\mathbb F}_2$ with two generators; 
 \end{itemize}
\item[{\rm(b)}] The discrete group $G_d$ is amenable if and only if 
                any of the following conditions is satisfied: 
 \begin{itemize}
 \item[{\rm(j)}] the group $G/R$ is finite 
                   (i.e., $G$ is a solvable Lie group); 
 \item[{\rm(jj)}] there exists no subgroup of $G_d$ isomorphic 
                   to ${\mathbb F}_2$. 
 \end{itemize}
\end{itemize}
}
\noindent We refer to Theorems (3.8) and (3.9) in \cite{Pa88} 
for proofs of these facts.
\qed
\end{example}

\begin{remark}\label{me6}
\normalfont
Assume that $\phi\colon\widetilde{G}\to G$ is a surjective homomorphism 
of topological groups such that $\Ker\phi$ is an abelian group. 
Then the group $\widetilde{G}$ is amenable if and only if $G$ is so. 
In fact, it follows by Example~\ref{me3} that the topological subgroup 
$\Ker\phi$ of $\widetilde{G}$ is amenable, 
and then the assertion follows for instance by remark $2^\circ)$ in \S 3 
of Expos\'e n$^o$~1 of \cite{Ey72}. 
\qed
\end{remark}

\subsubsection{Mimicking the group algebras of compact groups}

\begin{definition}\label{me7}
\normalfont 
Let $G$ be a topological group and consider the duality pairing 
$\langle\cdot,\cdot\rangle\colon(\RUCb(G))^*\times\RUCb(G)\to{\mathbb C}$. 
There exists a bounded bilinear map 
$$(\RUCb(G))^*\times\RUCb(G)\to\RUCb(G),\quad (\mu,\psi)\mapsto\mu\cdot\psi$$
defined by 
$(\mu\cdot\psi)(x)=\langle\mu,L_x\psi\rangle$ for all 
$x\in G$ (by (2.11) in \cite{Pa88}). 
The \emph{Arens-type product} on $(\RUCb(G))^*$ 
is the bounded bilinear mapping 
$$(\RUCb(G))^*\times(\RUCb(G))^*\to(\RUCb(G))^*,\quad 
(\mu,\nu)\mapsto\mu\cdot\nu$$
defined by 
$\langle \mu\cdot\nu,\psi\rangle:=\langle \mu, \nu\cdot\psi\rangle$ 
for all $\psi\in\RUCb(G)$. 

Similarly, by using the duality pairing 
$\langle\cdot,\cdot\rangle\colon(\LUCb(G))^*\times\LUCb(G)\to{\mathbb C}$,  
one defines a bounded bilinear map 
$$(\LUCb(G))^*\times\LUCb(G)\to\LUCb(G),\quad (\mu,\psi)\mapsto\mu\cdot\psi$$
by 
$(\mu\cdot\psi)(x)=\langle\mu,R_x\psi\rangle$ for all 
$x\in G$ (by the version of (2.11) in \cite{Pa88} for left uniformly continuous functions). 
The \emph{Arens-type product} on $(\LUCb(G))^*$ 
is the bounded bilinear mapping 
$$(\LUCb(G))^*\times(\LUCb(G))^*\to(\LUCb(G))^*,\quad 
(\mu,\nu)\mapsto\mu\cdot\nu$$
defined by 
$\langle \mu\cdot\nu,\psi\rangle:=\langle \nu, \mu\cdot\psi\rangle$ 
for all $\psi\in\LUCb(G)$. 
\qed
\end{definition}

\begin{remark}\label{Sigma}
\normalfont
Let $G$ be a topological group, consider the space $\Cb(G)$ 
consisting of the continuous elements of $\ell^\infty(G)$,  
and define 
$\sigma\colon\Cb(G)\to\Cb(G)$, $(\sigma(\psi))(x)=\overline{\psi(x^{-1})}$ 
whenever $x\in G$ and $\psi\in\Cb(G)$. 
The mapping $\sigma$ is an antilinear isometric $*$-endomorphism 
of the unital $C^*$-algebra $\Cb(G)$ which satisfies $\sigma^2=\id_{\Cb(G)}$ 
and has the following additional properties: 
\begin{itemize}
\item[{\rm(1)}] 
 For every $x\in G$ we have $L_x\circ\sigma=\sigma\circ R_{x^{-1}}$. 
\item[{\rm(2)}] 
 We have $\sigma(\RUCb(G))=\LUCb(G)$ and $\sigma(\LUCb(G))=\RUCb(G)$. 
\item[{\rm(3)}] 
  If we define $\Sigma\colon(\RUCb(G))^*\to(\LUCb(G))^*$ 
   as the anti-dual map of $\sigma\colon\LUCb(G)\to\RUCb(G)$, 
   that is, $\langle\Sigma(\mu),\phi\rangle=\overline{\langle\mu,\sigma(\phi)\rangle}$ 
   for all $\mu\in(\RUCb(G))^*$ and $\phi\in\LUCb(G)$, 
   then the diagram 
$$ \begin{CD} 
    (\RUCb(G))^*\times\RUCb(G) @>>> \RUCb(G) \\
    @V{\Sigma\times\sigma}VV @VV{\sigma}V \\
    (\LUCb(G))^*\times\LUCb(G) @>>> \LUCb(G)
 \end{CD}$$
   is commutative,where the horizontal arrows are 
   the maps introduced in Definition~\ref{me7}. 
\item[{\rm(4)}] The mapping $\Sigma\colon(\RUCb(G))^*\to(\LUCb(G))^*$ 
is antilinear, isometric, bijective, 
and for all $\mu_1,\mu_2\in(\RUCb(G))^*$ we have 
 $\Sigma(\mu_1\cdot\mu_2)=\Sigma(\mu_2)\cdot\Sigma(\mu_1)$. 
\end{itemize}
In fact, property (1) follows by a straightforward computation, 
and it implies property~(2) at once. 
For property~(3) note that for all $\mu\in(\RUCb(G))^*$ and $\psi\in\RUCb(G)$ 
we have 
$$\begin{aligned}
(\Sigma(\mu)\cdot\sigma(\psi))(x)
&=\langle\Sigma(\mu),R_x(\sigma(\psi))\rangle
=\langle\Sigma(\mu),\sigma(L_{x^{-1}}(\psi))\rangle \\
&=\overline{\langle\mu,\sigma^2(L_{x^{-1}}(\psi))\rangle}
=\overline{\langle\mu,L_{x^{-1}}(\psi)\rangle}
=\sigma(\mu\cdot\psi)(x)
\end{aligned}$$
whenever $x\in G$. 
To check property~(4), compute 
$$\begin{aligned}
\langle\Sigma(\mu_2)\cdot\Sigma(\mu_1),\varphi\rangle 
 &=\langle\Sigma(\mu_1),\Sigma(\mu_2)\cdot\varphi\rangle 
  \mathop{=}\limits^{(3)}\
   \langle\Sigma(\mu_1),\sigma(\mu_2\cdot\sigma(\varphi))\rangle \\
 &=\overline{\langle\mu_1,\mu_2\cdot\sigma(\varphi)\rangle}
  =\overline{\langle\mu_1\cdot\mu_2,\sigma(\varphi)\rangle}
  =\langle\Sigma(\mu_1\cdot\mu_2),\varphi\rangle
\end{aligned}$$
for every $\varphi\in\LUCb(G)$. 
\qed
\end{remark}

The facts described in the following theorem can be found 
in Theorems 2.5~and~4.1 of \cite{Gr}. 

\begin{theorem}\label{me8}
Every topological group $G$ has the following properties:  
\begin{itemize}
\item[{\rm(a)}] The Arens-type products 
make $(\RUCb(G))^*$ and $(\LUCb(G))^*$ 
into associative Banach algebras. 
\item[{\rm(b)}] Each continuous unitary representation 
$\pi\colon G\to\Bc(\Hc_\pi)$ gives rise to two representations 
of Banach algebras, 
$\widehat{\pi}_{\Rc}\colon(\RUCb(G))^*\to\Bc(\Hc_\pi)$ 
and 
$\widehat{\pi}_{\Lc}\colon(\LUCb(G))^*\to\Bc(\Hc_\pi)$, 
by means of the formulas 
\begin{equation}\label{integration}
(\widehat{\pi}_{\Rc}(\mu)\xi\mid\eta)
=\langle\mu,(\pi(\cdot)\xi\mid\eta)\rangle
\text{ and }
(\widehat{\pi}_{\Lc}(\nu)\xi\mid\eta)
=\langle\nu,(\pi(\cdot)\xi\mid\eta)\rangle 
\end{equation}
for all $\xi,\eta\in\Hc$, $\mu\in(\RUCb(G))^*$, and $\nu\in(\LUCb(G))^*$. 
These representations are related by the commutative diagram 
\begin{equation}\label{RL}
\begin{CD}
(\RUCb(G))^* @>{\widehat{\pi}_{\Rc}}>> \Bc(\Hc_\pi) \\
@V{\Sigma}VV @VV{S}V \\
(\LUCb(G))^* @>{\widehat{\pi}_{\Lc}}>> \Bc(\Hc_\pi)
\end{CD}
\end{equation}
where $S\colon\Bc(\Hc_\pi)\to\Bc(\Hc_\pi)$, $b\mapsto b^*$. 
\end{itemize}
\end{theorem}

\begin{proof}
Assertion~(a) follows at once by (2.8) and 
(the left-sided version of) (2.11) in \cite{Pa88}. 
See also Example~(19.23)(b) in \cite{HR63}, \cite{Te63}, and \cite{Bu50}. 

For assertion (b), firstly note that the matrix coefficients 
$\psi_{\xi,\eta}=(\pi(\cdot)\xi\mid\eta)$ 
belong to the function space $\UCb(G)=\RUCb(G)\cap\LUCb(G)$ 
for arbitrary $\xi,\eta\in\Hc_\pi$. 
To see this, just note that 
for all $x\in G$ we have $L_x(\psi_{\xi,\eta})=\psi_{\xi,\pi(x)^*\eta}$
and  
$R_x(\psi_{\xi,\eta})=\psi_{\pi(x)\xi,\eta}$, 
and then use the continuity of the representation~$\pi$. 
Thus the right-hand sides of both equalities 
in \eqref{integration} make sense, 
and then by means of the estimate 
$\Vert\psi_{\xi,\eta}\Vert_\infty\le\Vert\xi\Vert\Vert\eta\Vert$ 
we get $\widehat{\pi}_{\Rc}(\mu),\widehat{\pi}_{\Lc}(\nu)\in\Bc(\Hc_\pi)$ and 
$\Vert\widehat{\pi}_{\Rc}(\mu)\Vert\le\Vert\mu\Vert$ 
and $\Vert\widehat{\pi}_{\Lc}(\nu)\Vert\le\Vert\nu\Vert$ . 
In addition, since $\sigma(\psi_{\xi,\eta})=\psi_{\eta,\xi}$, 
we get 
$(\widehat{\pi}_{\Lc}(\Sigma(\mu))\xi\mid\eta)
=\langle\Sigma(\mu),\psi_{\xi,\eta}\rangle
=\overline{\langle\mu,\sigma(\psi_{\xi,\eta})\rangle}
=\overline{\langle\mu,\psi_{\eta,\xi}\rangle}
=\overline{(\widehat{\pi}_{\Rc}(\mu)\eta\mid\xi)}
=(\widehat{\pi}_{\Rc}(\mu)^*\xi\mid\eta)
$, 
whence $\widehat{\pi}_{\Lc}(\Sigma(\mu))=\widehat{\pi}_{\Rc}(\mu)^*$, 
and thus the diagram \eqref{RL} is commutative. 
To conclude the proof of (b), let $\mu_1,\mu_2\in(\RUCb(G))^*$. 
Then 
$(\widehat{\pi}_{\Rc}(\mu_1\cdot\mu_2)\xi\mid\eta)
=\langle\mu_1\cdot\mu_2,\psi_{\xi,\eta}\rangle
=\langle\mu_1,\mu_2\cdot\psi_{\xi,\eta}\rangle
=\langle\mu_1,\psi_{\widehat{\pi}_{\Rc}(\mu_2)\xi,\eta}\rangle
=(\widehat{\pi}_{\Rc}(\mu_1)\widehat{\pi}_{\Rc}(\mu_2)\xi\mid\eta)
$, where the next-to-last equality holds 
since for every $x\in G$ we have 
$(\mu_2\cdot\psi_{\xi,\eta})(x)
=\langle\mu_2,L_x(\psi_{\xi,\eta})\rangle
=\langle\mu_2,\psi_{\xi,\pi(x)^*\eta}\rangle
=(\widehat{\pi}_{\Rc}(\mu_2)\xi\mid \pi(x)^*\eta)
=(\pi(x)\widehat{\pi}_{\Rc}(\mu_2)\xi\mid \eta) 
=\psi_{\widehat{\pi}_{\Rc}(\mu_2)\xi,\eta}(x)$. 
Thus $\widehat{\pi}_{\Rc}$ is an algebra representation. 
Similarly, for $\nu_1,\nu_2\in(\LUCb(G))^*$ we can check that 
$\nu_1\cdot\psi_{\xi,\eta}=\psi_{\xi,\widehat{\pi}_{\Lc}(\nu_1)^*\eta}$, 
whence as above we get 
$(\widehat{\pi}_{\Lc}(\nu_1\cdot\nu_2)\xi\mid\eta)
=(\widehat{\pi}_{\Lc}(\nu_1)\widehat{\pi}_{\Lc}(\nu_2)\xi\mid\eta)
$, 
and the proof ends.
\end{proof}

\begin{remark}\label{me9}
\normalfont
In the setting of Theorem~\ref{me8}, 
$\RUCb(G)$ and $\LUCb(G)$ are 
commutative unital isomorphic $C^*$-algebras, 
hence there exist $*$-isomorphisms 
$\RUCb(G)\simeq\Cc(\Mg_0(G))\simeq\LUCb(\Mg_0(G))$, 
where $\Mg_0(G)$ is a compact topological space. 
In the special case when the group $G$ is compact we have 
$\Mg_0(G)=G$, $\Cc(G)=\RUCb(G)=\LUCb(G)=\UCb(G)$, 
and the involutive Banach algebra 
$(\UCb(G))^*$ 
is the convolution measure algebra 
of $G$. 
In the general case, the topological duals of $\RUCb(G)$ and $\LUCb(G)$ 
still consist of the (not necessarily positive) measures 
on the compact space  $\Mg_0(G)$, however the Arens-type products 
may not be defined by convolution formulas 
by reasonable convolution formulas involving $\Mg_0(G)$ 
(cf.~the comment preceding Proposition~(2.25) in \cite{Pa88}). 
See however the general theory of convolutions of functionals 
developed in Chapter~5 of \cite{HR63} or the other references mentioned 
in connection with the proof of Theorem~\ref{me8}(a) above. 
And another problem to be dealt with is 
to find a method to distinguish in the set of pairs of 
representations  of the Banach algebras $(\RUCb(G))^*$ and $(\LUCb(G))^*$ 
that make the diagram \eqref{RL} commutative, 
the ones that come from unitary representations of $G$ 
by means of~\eqref{integration}. 
In this connection, let us recall that there exists a promising 
approach to an axiomatic theory of group algebras 
by means of the host algebras, that is,  
$C^*$-algebras whose representations correspond in a one-to-one fashion  
with the unitary representations of a given group; 
see \cite{Gr05}, \cite{GN07}, 
and the references therein. 
\qed
\end{remark}

\begin{remark}\label{me11}
\normalfont
Let $G$ be an amenable topological group and pick 
a left invariant mean
$\mu\colon\RUCb(G)\to{\mathbb C}$. 
Then $\mu$ is a state of the $C^*$-algebra $\RUCb(G)$, 
and the corresponding 
Gelfand-Naimark-Segal construction 
leads to a cyclic $*$-representation 
$\iota_\mu\colon\RUCb(G)\to\Bc(\Hc^{(\mu)})$. 
Recall that the Hilbert space $\Hc^{(\mu)}$ is obtained 
out of $\RUCb(G)$ as a quotient followed by a completion 
with respect to the non-negative definite, sesquilinear form 
$\RUCb(G)\times\RUCb(G)\to{\mathbb C}$, 
$(\psi,\chi)\mapsto\mu(\psi\chi^*)$. 
For each $x\in G$ we have 
$\mu\circ L_{x^{-1}}|_{\RUCb(G)}=\mu$, 
hence the mapping $\RUCb(G)\to \RUCb(G)$, $\psi\mapsto L_{x^{-1}}\psi$ 
induces a unitary representation $\lambda_\mu\colon G\to\Bc(\Hc^{(\mu)})$, 
which is easily seen to be continuous. 
In the special case when $G$ is compact and $\mu$ is 
the probability Haar measure on $G$, we have $\Hc^{(\mu)}=L^2(G,\mu)$ 
and $\lambda_\mu$ is the regular representation of $G$. 
For this reason, in the general case  of an amenable group $G$, 
one can think of $\lambda_\mu$ 
as a \emph{regular representation associated with 
the left invariant mean~$\mu$}. 

We have to point out that it may happen that $\dim\Hc^{(\mu)}=1$. 
For instance, this is the case if $G$ is an extremely amenable group 
and the left invariant mean $\mu$ is chosen to be multiplicative. 
(See \cite{Pe06} for specific examples and for details on the latter notions.) 
It may also happen that the regular representation is trivial, 
in the sense that $\lambda_\mu(x)=\1$ for all $x\in G$. 
This is the case for the exotic Banach-Lie groups (see \cite{Ba83}), 
which are abelian topological groups that 
do not have any non-trivial continuous representation. 
It is not difficult to verify that 
the regular representation $\lambda_\mu$ is trivial if and only if 
$\mu(\chi\psi)=\mu(\chi L_x\psi)$ for all $\chi,\psi\in\RUCb(G)$ and $x\in G$. 
\qed
\end{remark}

\section{Lifting group decompositions to covering groups}

This section has a technical character and its main purpose 
is to provide tools for enriching the class of reductive Banach-Lie groups 
to be set forth in the next section. 
In the proof of the following statements we use some ideas 
from the proofs of Theorem 6.31~and~6.46 in \cite{Kn96}. 

\begin{lemma}\label{LL1}
Assume that $G$ and $\widetilde{G}$ are Banach-Lie groups, 
and $e\colon\widetilde{G}\to G$ is a covering homomorphism. 
Let $K$ be any Banach-Lie subgroup of $G$, denote 
$\widetilde{K}:=e^{-1}(K)$,  
and define 
$\psi\colon\widetilde{G}/\widetilde{K}\to G/K$, 
$\widetilde{g}\widetilde{K}\mapsto e(\widetilde{g})K$. 
Then the following assertions hold: 
\begin{itemize}
\item[{\rm(i)}] $\widetilde{K}$ is a Banach-Lie subgroup of $\widetilde{G}$ 
and the mapping $e|_{\widetilde{K}}\colon\widetilde{K}\to K$ 
is a covering homomorphism. 
\item[{\rm(ii)}] The mapping $\psi$ is a well-defined diffeomorphism. 
\item[{\rm(iii)}] If the group $\widetilde{G}$ is connected 
and the smooth homogeneous space $G/K$ is simply connected, 
then both $\widetilde{K}$ and $K$ are connected.  
\end{itemize}
\end{lemma}

\begin{proof}
Since $\Lie(e)=T_{\1}e\colon\Lie(\widetilde{G})\to\Lie(G)$ 
is an isomorphism of Banach-Lie algebras, 
it follows by Proposition~4.8 in \cite{Be06} 
that $\widetilde{K}$ 
is a Banach-Lie subgroups of $G$ and the tangent map 
$\Lie(e)|_{T_{\1}\widetilde{K}}\colon T_{\1}\widetilde{K}\to T_{\1}K$. 
Then Remark~C.13(b) in \cite{Be06} shows that 
$e|_{\widetilde{K}}\colon\widetilde{K}\to X$ 
is a covering map. 
This completes the proof of assertion~(i). 

To prove assertion~(ii), 
note that $\psi$ is injective since $\widetilde{K}=e^{-1}(K)$, 
and $\psi$ is surjective because $e$ is so. 
On the other hand, we have a commutative diagram 
$$\begin{CD}
\widetilde{G} @>{e}>> G\\
@VVV @VVV \\
\widetilde{G}/\widetilde{K} @>{\psi}>> G/K
\end{CD}
$$ 
where the vertical arrows (which are the quotient maps) 
are submersions. 
Since the covering map $e$ is a local diffeomorphism, 
it follows from this commutative diagram that 
$\psi$ is a local diffeomorphism as well.  
Then $\psi$ is actually a diffeomorphism since 
we have already seen that it is bijective.  

For assertion~(iii), recall from~(i) that $\widetilde{K}$ 
is a covering group of $K$, so it will be enough to show that 
$\widetilde{K}$ is connected. 
And the latter property follows from 
the long exact sequence of homotopy groups 
$$0\leftarrow\pi_0(\widetilde{G}/\widetilde{K})\leftarrow
\pi_0(\widetilde{G})\leftarrow\pi_0(\widetilde{K})
\leftarrow\pi_1(\widetilde{G}/\widetilde{K})\leftarrow\cdots $$
since $\pi_0(\widetilde{G})=\{0\}$ by the assumption on $\widetilde{G}$, 
while $\pi_1(\widetilde{G}/\widetilde{K})=\{0\}$ 
by the assumption on $G/K$ along with the fact that 
$G/K$ is homeomorphic to $\widetilde{G}/\widetilde{K}$ 
according to assertion~(ii). 
\end{proof}

We now come to a proposition to the effect that 
the Cartan decompositions lift to the covering groups. 

\begin{proposition}\label{LL2}
Let $G$ and $\widetilde{G}$ be to Banach-Lie groups, 
and assume that $e\colon\widetilde{G}\to G$ is a covering homomorphism. 
Now let $K$ be any Banach-Lie subgroup of $G$ and denote 
$\widetilde{K}:=e^{-1}(K)$. 
Denote $\Lie(G)=\gg$ and $\Lie(K)=\kg$, 
and assume that $\pg$ is a closed linear subspace 
of $\gg$ such that $\gg=\kg\dotplus\pg$ and the mapping 
$\varphi\colon K\times\pg\to G$, 
$(k,X)\mapsto k\cdot\exp_GX$
is a diffeomorphism. 
Moreover denote 
$\Lie(\widetilde{G}):=\widetilde{\gg}$, 
$\Lie(\widetilde{K}):=\widetilde{\kg}$, 
and $\widetilde{\pg}:=\Lie(e)^{-1}(\pg)$. 
Then the mapping 
$\widetilde{\varphi}\colon
\widetilde{K}\times\widetilde{\pg}\to\widetilde{G}$, 
$(\widetilde{k},\widetilde{X})\mapsto 
\widetilde{k}\cdot\exp_{\widetilde{G}}\widetilde{X}
$
is a diffeomorphism as well. 
\end{proposition}

\begin{proof}
First note that $\widetilde{K}$ is a Banach-Lie subgroup of $\widetilde{G}$ 
by Lemma~\ref{LL1}, 
and $\Lie(e)\colon\widetilde{\gg}\to\gg$ is an isomorphism of 
Banach-Lie algebras, so that 
$\Lie(e)\widetilde{\kg}=\kg$ and 
$\widetilde{\gg}=\widetilde{\kg}\dotplus\widetilde{\pg}$.  
Now note that there exists a commutative diagram 
$$\begin{CD}
\widetilde{K}\times\widetilde{\pg} @>{\widetilde{\varphi}}>> \widetilde{G}\\
@V{e|_{\widetilde{K}}\times\Lie(e)|_{\widetilde{\pg}}}VV @VV{e}V \\
K\times\pg @>{\varphi}>> G
\end{CD}
$$
whose vertical arrows are covering maps (see Lemma~\ref{LL1}(i)). 
Since $\varphi$ is a diffeomorphism by assumption, 
it then follows that $\widetilde{\varphi}$ is a local diffeomorphism. 
To get the wished-for conclusion, 
we still have to prove that $\widetilde{\varphi}$ is a bijective map. 

To check that $\widetilde{\varphi}$ is injective, 
let $\widetilde{k}_j\in\widetilde{K}_j$ 
and $\widetilde{X}_j\in\widetilde{\pg}_j$ for $j=1,2$ 
such that 
$\widetilde{k}_1\cdot\exp_{\widetilde{G}}\widetilde{X}_1
=\widetilde{k}_2\cdot\exp_{\widetilde{G}}\widetilde{X}_2$. 
By applying the map $e$ to both sides of the latter equality, 
and using the commutation relation between the exponential maps 
and group homomorphisms (see e.g., Remark~2.34 in \cite{Be06}), 
we get 
$e(\widetilde{k}_1)\cdot\exp_{G}(\Lie(e)\widetilde{X}_1)
=e(\widetilde{k}_2)\cdot\exp_{G}(\Lie(e)\widetilde{X}_2)$. 
Since $\varphi$ is injective, it then follows that 
$\Lie(e)\widetilde{X}_1=\Lie(e)\widetilde{X}_2$ and
$e(\widetilde{k}_1)=e(\widetilde{k}_2)$.
The first of these equalities implies that 
$\widetilde{X}_1=\widetilde{X}_2$, whence 
$\widetilde{k}_1=\widetilde{k}_2$ by the assumption on 
$\widetilde{k}_j\in\widetilde{K}_j$ 
and $\widetilde{X}_j\in\widetilde{\pg}_j$ for $j=1,2$. 
Now, to prove that $\widetilde{\varphi}$ is surjective, 
let $\widetilde{g}\in\widetilde{G}$ arbitrary. 
Then $e(\widetilde{g})\in G$ 
hence there exist $k\in K$ and $X\in\pg$ such that 
$e(\widetilde{g})=k\cdot\exp_GX$ 
since $\varphi$ is surjective. 
Further on, pick $\widetilde{k}_0\in e^{-1}(k)$ arbitrary 
and denote $\widetilde{X}:=\Lie(e)^{-1}X$. 
Then $e(\widetilde{g})
=e(\widetilde{k}_0)\cdot\exp_{G}(\Lie(e)\widetilde{X})
=e(\widetilde{k}_0\cdot\exp_{\widetilde{G}}\widetilde{X})$, 
so that by denoting 
$\widetilde{z}:=\widetilde{k}_0\cdot\exp_{\widetilde{G}}\widetilde{X}
\cdot\widetilde{g}^{-1}$ 
we have $\widetilde{z}\in e^{-1}(\1)\subseteq\widetilde{K}$. 
Thus $\widetilde{k}:=\widetilde{z}^{-1}\widetilde{k}_0\in\widetilde{K}$ 
and we have 
$\widetilde{g}=\widetilde{k}\cdot\exp_{\widetilde{G}}\widetilde{X}$, 
which concludes the proof.
\end{proof}

The next proposition shows that the familiar integration of 
local Cartan involutions to global Cartan involutions 
carries over to the setting of Banach-Lie groups. 

\begin{proposition}\label{LL3}
Let $G$ be a connected Banach-Lie group with the Lie algebra 
$\Lie(G)=\gg$, $K$ a Banach-Lie subgroup of $G$, 
and assume that there exists an automorphism 
$\theta\in\Aut(\gg)$ such that $\theta^2=\id_{\gg}$, 
$\Lie(K)=\Ker(\theta-\id_{\gg})$, and 
the mapping 
$\varphi\colon K\times\pg\to G$,
$(k,X)\mapsto k\cdot\exp_GX$
is a diffeomorphism. 
Then there exists a unique automorphism 
$\Theta\in\Aut(G)$ such that 
$\Lie(\Theta)=\theta$ and 
$K=\{g\in G\mid\Theta(g)=g\}$. 
\end{proposition}

\begin{proof}
Let $e\colon\widetilde{G}\to G$ be the universal covering 
of $G$, 
and denote 
$\kg=\Lie(K)$, 
$\pg:=\Ker(\theta+\id_{\gg})$,
$\widetilde{K};=e^{-1}(K)$, 
$\widetilde{\gg}:=\Lie(\widetilde{G})$, 
$\widetilde{\pg}:=\Lie(e)^{-1}(\pg)$, 
and $\widetilde{\theta}:=\Lie(e)^{-1}\circ\theta\circ\Lie(e)\in
\Aut(\widetilde{\gg})$. 
Then Proposition~\ref{LL2} shows that the mapping 
$\widetilde{\varphi}\colon
\widetilde{K}\times\widetilde{\pg}\to\widetilde{G}$, 
$(\widetilde{k},\widetilde{X})\mapsto 
\widetilde{k}\cdot\exp_{\widetilde{G}}\widetilde{X}$ 
is a diffeomorphism. 

On the other hand, since the group $\widetilde{G}$ is connected and 
simply connected, it follows that there exists a unique smooth homomorphism 
$\widetilde{\Theta}\colon\widetilde{G}\to\widetilde{G}$ 
such that $\Lie(\widetilde{\Theta})=\widetilde{\theta}$. 
(See for instance Remark~3.13 in \cite{Be06}.) 
Since $\widetilde{\theta}^2=\id_{\widetilde{\gg}}$, 
it then follows that $\widetilde{\Theta}^2=\id_{\widetilde{G}}$, 
and in particular $\widetilde{\Theta}\in\Aut(\widetilde{G})$. 

Now we use hypothesis~(iii) to see that the mapping 
$\tau\colon\pg\to G/K$, $X\mapsto(\exp_GX)^{-1}K$ 
is a diffeomorphism. 
In fact, it is clear that this is a smooth map, 
and its inverse is the smooth well-defined map 
$G/K\to\pg$, $gK\mapsto\pr_{\pg}(g^{-1})$, 
where $\pr_{\pg}\colon G\to\pg$ is the projection onto $\pg$ 
defined by means of the diffeomorphism 
$\varphi^{-1}\colon G\to K\times\pg$. 

Thus $\tau\colon\pg\to G/K$ is a diffeomorphism, 
and in particular $G/K$ is simply connected since $\pg$ is so. 
Then Lemma~\ref{LL1} shows that both $K$ and $\widetilde{K}$ 
are connected. 
Since $\widetilde{\theta}|_{\widetilde{\kg}}=\id_{\widetilde{\kg}}$ 
and $\widetilde{K}$ is connected, 
it follows that $\widetilde{\Theta}|_{\widetilde{K}}=\id_{\widetilde{K}}$. 
In particular, the subgroup $e^{-1}(\1)$ is invariant under 
$\widetilde{\Theta}$ (since $e^{-1}(\1)\subseteq\widetilde{K}$). 
Then there exists a unique automorphism 
$\Theta\in\Aut(G)$ such that $\Theta\circ e=e\circ\widetilde{\Theta}$. 
In addition, since $\widetilde{\Theta}^2=\widetilde{\Theta}$, 
it follows that $\Theta^2=\Theta$. 

It remains to prove that 
$K=\{g\in G\mid\Theta(g)=g\}$. 
The inclusion $\subseteq$ is clear since 
$\widetilde{\Theta}|_{\widetilde{K}}=\id_{\widetilde{K}}$ 
and $K=\widetilde{K}/e^{-1}(\1)$. 
Conversely, 
let $g\in G$ such that $\Theta(g)=g$. 
By hypothesis~(iii), 
there exist $k\in K$ and $X\in\pg$ such that $g=k\cdot\exp_GX$. 
Then 
$k\cdot\exp_GX=g=\Theta(g)=\Theta(k)\cdot\exp_G(\theta(-X))
=k\cdot\exp_G(-X)$,
so that $\exp_G(2X)=\1$. 
Thus $\varphi(\1,2X)=\varphi(\1,0)$, and then $X=0$ since 
$\varphi\colon K\times\pg\to G$ is injective. 
Consequently $g=k\in K$, and the proof ends. 
\end{proof}

We now come to a proposition that allows us to lift 
the global Iwasawa decompositions to covering groups. 

\begin{proposition}\label{lift_iw}
Let $G$ be a connected Banach-Lie group, and $K$, $A$, and $N$ 
connected Banach-Lie subgroups of $G$ such that the multiplication map
$\m\colon K\times A\times N\to G$
is a diffeomorphism. 
In addition, assume that $A$ and $N$ are simply connected 
and $AN=NA$. 

Now assume that we have a connected Banach-Lie group $\widetilde{G}$ 
with a covering homomorphism 
$e\colon\widetilde{G}\to G$, and define 
$\widetilde{K}:=e^{-1}(K)$, $\widetilde{A}:=e^{-1}(A)$, 
and $\widetilde{N}:=e^{-1}(N)$. 
Then $\widetilde{K}$, $\widetilde{A}$, and $\widetilde{N}$ 
are connected Banach-Lie subgroups of $\widetilde{G}$ and the multiplication 
map 
$\widetilde{\m}\colon\widetilde{K}\times\widetilde{A}\times\widetilde{N}\to 
\widetilde{G}$
is a diffeomorphism. 
\end{proposition}

\begin{proof}
It follows by Lemma~\ref{LL1}(i)
that $\widetilde{K}$, $\widetilde{A}$, and $\widetilde{N}$ 
are Banach-Lie subgroups of $G$ and 
$e|_{\widetilde{X}}\colon\widetilde{X}\to X$ 
is a covering map when $X$ is either of the groups $K$, $A$, $N$ or $G$. 
Since the groups $A$ and $N$ are simply connected, 
it then follows 
that the map $e|_{\widetilde{X}}\colon\widetilde{X}\to X$ 
is actually a diffeomorphism if $X=A$ or $X=N$ 
(see for instance Theorem~C.18 in \cite{Be06}). 
In particular, the groups $\widetilde{A}$ and $\widetilde{N}$ 
are connected and simply connected. 

To prove that $\widetilde{K}$ is connected as well, we 
first show that $G/K$ is homeomorphic to the simply connected 
group $B:=AN\simeq A\times N$, where the diffeomorphism 
$A\times N\to B$ is defined by the multiplication in $G$, 
as an easy consequence of the assumption. 
Now let $\pr_K\colon G\to K$ and $\pr_B\colon G\to B$ 
be the smooth projections given by the inverse of the diffeomorphism 
$K\times B\to G$. 
Then the continuous map 
$\tau\colon B\to G/K$, $b\mapsto b^{-1}K$ 
is bijective and its inverse is also continuous 
since it is given by 
$\tau^{-1}\colon G/K\to B$, $gK\mapsto\pr_B(g^{-1})$. 
Thus $\tau\colon B\to G/K$ is a homeomorphism. 
By taking into account the homeomorphism 
$\psi\colon\widetilde{G}/\widetilde{K}\to G/K$ 
provided by Lemma~\ref{LL1}(ii), 
it then follows that the quotient 
$\widetilde{G}/\widetilde{K}$ is simply connected, 
since $B$ ($\simeq A\times N$) is simply connected. 
Then Lemma~\ref{LL1}(iii) shows that $\widetilde{K}$ is connected.

We now prove the assertion regarding the multiplication map 
$\widetilde{\m}\colon\widetilde{K}\times\widetilde{A}\times\widetilde{N}\to 
\widetilde{G}$. 
Using locally defined inverses of the covering map 
$e\colon\widetilde{G}\to G$, 
we see that $\widetilde{\m}$ is a local diffeomorphism 
at every point of $\widetilde{K}\times\widetilde{A}\times\widetilde{N}$. 
It remains to prove that $\widetilde{\m}$ is bijective. 
To prove that $\widetilde{\m}$ is surjective, 
let $\widetilde{g}\in\widetilde{G}$ 
arbitrary. 
Since $m\colon K\times A\times N\to G$ is surjective, 
there exist $k\in K$, $a\in A$, and $n\in N$ such that 
$e(\widetilde{g})=kan$. 
Denote $\widetilde{a}:=(e|_{\widetilde{A}})^{-1}(a)\in A$ and 
$\widetilde{n}:=(e|_{\widetilde{N}})^{-1}(n)\in N$, 
and take $\widetilde{k}_0\in e^{-1}(k)$ arbitrary, so that 
$e(\widetilde{k}_0\widetilde{a}\widetilde{n})=kan=e(\widetilde{g})$. 
Denoting 
$\widetilde{z}:=\widetilde{k}_0\widetilde{a}\widetilde{n}\widetilde{g}^{-1}$, 
we have $e(\widetilde{z})=1$, 
so $\widetilde{z}\in e^{-1}(\1)\subseteq e^{-1}(K)=\widetilde{K}$. 
Then $\widetilde{k}:=\widetilde{z}^{-1}\widetilde{k}_0\in\widetilde{K}$ 
and $\widetilde{g}=\widetilde{k}\widetilde{a}\widetilde{n}
\in\widetilde{K}\widetilde{A}\widetilde{N}$. 
Since $\widetilde{g}\in\widetilde{G}$ is arbitrary, 
it follows that 
$\widetilde{\m}\colon\widetilde{K}\times\widetilde{A}\times\widetilde{N}\to 
\widetilde{G}$ is surjective. 

To prove that $\widetilde{\m}$ is injective, 
first recall that $e|_{\widetilde{X}}\colon\widetilde{X}\to X$ 
is a bijective homomorphism if $X=A$ or $X=N$. 
Then the hypothesis that $AN=NA$ implies that 
$\widetilde{B}:=\widetilde{A}\widetilde{N}=\widetilde{N}\widetilde{A}$ 
is a subgroup of $\widetilde{G}$, and the multiplication map 
$\widetilde{A}\times\widetilde{N}\to\widetilde{B}$, 
$(\widetilde{a},\widetilde{n})\to\widetilde{a}\widetilde{n}$ 
is a bijection. 
Now assume that $\widetilde{k}_j\in\widetilde{K}$, 
$\widetilde{a}_j\in\widetilde{A}$, and $\widetilde{n}_j\in\widetilde{N}$ 
for $j=1,2$, and 
$\widetilde{k}_1\widetilde{a}_1\widetilde{n}_1
=\widetilde{k}_2\widetilde{a}_2\widetilde{n}_2$. 
Then  
$\widetilde{k}_2^{-1}\widetilde{k}_1^{-1}
=\widetilde{a}_2\widetilde{n}_2(\widetilde{a}_1\widetilde{n}_1)^{-1}
\in\widetilde{K}\cap\widetilde{B}$, 
so that it will be enough to check that 
$\widetilde{K}\cap\widetilde{B}=\{\1\}$, 
that is, $\widetilde{K}\cap\widetilde{A}\widetilde{N}=\{\1\}$. 
In order to prove the latter equality, 
let $\widetilde{x}\in\widetilde{K}\cap\widetilde{A}\widetilde{N}$ arbitrary. 
Then $e(\widetilde{x})\in K\cap AN=\{\1\}$. 
On the other hand, we have $\widetilde{x}=\widetilde{a}\widetilde{n}$ 
for some $\widetilde{a}\in\widetilde{A}$ and $\widetilde{n}\in\widetilde{N}$, 
and $e(\widetilde{a})e(\widetilde{n})=\1$. 
Since $A\cap N=\{\1\}$, it follows that 
$e(\widetilde{a})=e(\widetilde{n})=\1$. 
Using the fact that $e|_{\widetilde{X}}\colon\widetilde{X}\to X$ 
is a bijective homomorphism if $X=A$ or $X=N$, 
it then follows that $\widetilde{a}=\widetilde{n}=\1$, 
whence $\widetilde{x}=\widetilde{a}\widetilde{n}=\1$. 
Thus $\widetilde{K}\cap\widetilde{A}\widetilde{N}=\{\1\}$, 
and this completes the proof of the fact that the multiplication map 
$\widetilde{\m}\colon\widetilde{K}\times\widetilde{A}\times\widetilde{N}\to 
\widetilde{G}$ 
is bijective.  
\end{proof}

\section{What a reductive Banach-Lie group could be}\label{Sect5}

\noindent\textbf{Reductivity relative to a symmetric norming function.}
This notion is suggested by 
Remarks \ref{red_alg}~and~\ref{red_gr}. 

\begin{definition}\label{Phi_red}
\normalfont
Let $\Phi$ be a symmetric norming function. 
By \emph{$\Phi$-reductive Lie algebra} 
we mean any closed real Lie subalgebra of $\Sg_\Phi(\Hc)$ 
satisfying the following conditions: 
\begin{itemize}
\item[{\rm(i)}] for every $X\in\gg$ we have $X^*\in\gg$; 
\item[{\rm(ii)}] the set $\gg\cap\Fg$ of finite-rank operators in $\gg$ 
is dense in $\gg$ with respect to the norm $\Vert\cdot\Vert_\Phi$. 
Thus $\gg\subseteq\Sg_\Phi^{(0)}$ actually. 
\end{itemize}
In this case the connected Banach-Lie group $G$ 
($\subseteq \1+\Sg_\Phi^{(0)}(\Hc)$) 
corresponding 
to the Lie subalgebra $\gg$ of $\Sg_\Phi^{(0)}(\Hc)$ is said to be a 
\emph{$\Phi$-reductive linear Banach-Lie group}. 
By \emph{$\Phi$-reductive Banach-Lie group} we shall mean any covering group 
of a $\Phi$-reductive linear Lie group.

In the above setting, the closure of $\gg\cap\Fg$ with respect to 
the Hilbert-Schmidt norm $\Vert\cdot\Vert_2$ will be called 
the \emph{$L^*$-algebra associated with $\gg$} and will be denoted by~$\gg_2$.
\qed
\end{definition}

\begin{example}\label{classical} 
\normalfont
(See \cite{dlH72} and also Definitions 3.2~and~3.2 in \cite{Be07}.) 
Let $\Phi$ be any symmetric norming function and 
denote by $\GL(\Hc)$ the group of all invertible bounded linear operators 
on the complex Hilbert space $\Hc$.  
The \emph{classical complex Banach-Lie groups and  
Banach-Lie algebras} associated with $\Phi$ are defined as follows:   
\begin{itemize}
\item[{\rm(A)}]
$\GL_\Phi(\Hc)=\GL(\Hc)\cap(\1+\Sg_\Phi^{(0)}(\Hc))$ 
with the Lie algebra 
$\gl_\Phi(\Hc)=\Sg_\Phi^{(0)}(\Hc)$;
\item[{\rm(B)}]
$\OO_\Phi(\Hc)=\{g\in\GL_\Phi(\Hc)\mid g^{-1}=Jg^*J^{-1}\}$ 
(where  $J$ is a \emph{conjugation}, i.e., 
an antilinear isometry with $J^2=\1$)
with the Lie algebra 
$\og_\Phi(\Hc)=\{x\in\Sg_\Phi^{(0)}(\Hc)\mid x=-Jx^*J^{-1}\}$; 
\item[{\rm(C)}]
$\Sp_\Phi(\Hc)=\{g\in\GL_\Phi(\Hc)\mid 
 g^{-1}=\widetilde{J}g^*\widetilde{J}^{-1}\}$ 
(where $\widetilde{J}$ an \emph{anti-conjugation}, 
i.e., an antilinear isometry with 
$\widetilde{J}^2=-\1$), 
with the Lie algebra 
$\sp_\Phi(\Hc)=\{x\in\Sg_\Phi^{(0)}(\Hc)\mid 
x=-\widetilde{J}x^*\widetilde{J}^{-1}\}$. 
\end{itemize}

The \emph{classical real Banach-Lie groups and Banach-Lie algebras}   
associated with the symmetric norming function $\Phi$ are
the following: 
\begin{itemize}
\item[{\rm(AI)}]
$\GL_\Phi(\Hc;{\mathbb R})=\{g\in\GL_\Phi(\Hc)\mid gJ=Jg\}$ 
with the Lie algebra 
$\gl_\Phi(\Hc;{\mathbb R})
=\{x\in\Sg_\Phi^{(0)}(\Hc)\mid xJ=Jx\}$,  
where $J\colon\Hc\to\Hc$ is a conjugation; 
\item[{\rm(AII)}]
$\GL_\Phi(\Hc;{\mathbb H})=\{g\in\GL_\Phi(\Hc)\mid 
 g\widetilde{J}=\widetilde{J}g\}$ 
with the Lie algebra 
$\gl_\Phi(\Hc;{\mathbb H})
=\{x\in\Sg_\Phi^{(0)}(\Hc)\mid x\widetilde{J}=\widetilde{J}x\}$,  
where $\widetilde{J}\colon\Hc\to\Hc$ is an anti-conjugation;
\item[{\rm(AIII)}]
$\U_\Phi(\Hc_{+},\Hc_{-})=\{g\in\GL_\Phi(\Hc)\mid g^*Vg=V\}$ 
with the Lie algebra 
$\ug_\Phi(\Hc_{+},\Hc_{-})
=\{x\in\Sg_\Phi^{(0)}(\Hc)\mid x^*V=-Vx\}$, 
where $\Hc=\Hc_{+}\oplus\Hc_{-}$ and 
$V=\begin{pmatrix} \hfill 1 & \hfill 0 \\ \hfill 0 & \hfill -1\end{pmatrix}$ 
with respect to 
this orthogonal decomposition of $\Hc$; 
\item[{\rm(BI)}]
$\OO_\Phi(\Hc_{+},\Hc_{-})=\{g\in\GL_\Phi(\Hc)\mid g^{-1}=Jg^*J^{-1} 
\text{ and }g^*Vg=V\}$  
with 
$\og_\Phi(\Hc_{+},\Hc_{-})
=\{x\in\Sg_\Phi^{(0)}(\Hc)\mid x=-Jx^*J^{-1}\text{ and }x^*V=-Vx\}$, 
where $\Hc=\Hc_{+}\oplus\Hc_{-}$, 
$V=\begin{pmatrix} \hfill 1 & \hfill 0 \\ \hfill 0 & \hfill -1\end{pmatrix}$ 
with respect to 
this decomposition of $\Hc$, 
and $J\colon\Hc\to\Hc$ is a conjugation 
such that $J(\Hc_{\pm})\subseteq\Hc_{\pm}$; 
\item[{\rm(BII)}]
$\OO^*_\Phi(\Hc)=\{g\in\GL_\Phi(\Hc)\mid g^{-1}=Jg^*J^{-1} 
\text{ and }g\widetilde{J}=\widetilde{J}g\}$ 
with the Lie algebra 
$\og^*_\Phi(\Hc)
=\{x\in\Sg_\Phi^{(0)}(\Hc)\mid 
x=-Jx^*J^{-1}\text{ and }x\widetilde{J}=\widetilde{J}x\}$, 
where  $J\colon\Hc\to\Hc$ is a conjugation and 
$\widetilde{J}\colon\Hc\to\Hc$ is an anti-conjugation
such that $J\widetilde{J}=\widetilde{J}J$; 
\item[{\rm(CI)}]
$\Sp_\Phi(\Hc;{\mathbb R})=\{g\in\GL_\Phi(\Hc)\mid 
g^{-1}=\widetilde{J}g^*\widetilde{J}^{-1}\text{ and }gJ=Jg\}$ 
with the Lie algebra 
$\sp_\Phi(\Hc;{\mathbb R})=\{x\in\Sg_\Phi^{(0)}(\Hc)\mid 
-x=\widetilde{J}x^*\widetilde{J}^{-1}\text{ and }xJ=Jx\}$, 
where $\widetilde{J}\colon\Hc\to\Hc$ is any anti-conjugation and  
$J\colon\Hc\to\Hc$ is any conjugation such that 
$J\widetilde{J}=\widetilde{J}J$; 
\item[{\rm(CII)}]
$\Sp_\Phi(\Hc_{+},\Hc_{-})=\{g\in\GL_\Phi(\Hc)\mid 
g^{-1}=\widetilde{J}g^*\widetilde{J}^{-1}\text{ and }g^*Vg=V\}$ 
with 
$\sp_\Phi(\Hc_{+},\Hc_{-})
=\{x\in\Sg_\Phi^{(0)}(\Hc)\mid 
x=-\widetilde{J}x^*\widetilde{J}^{-1}\text{ and }x^*V=-Vx\}$, 
where $\Hc=\Hc_{+}\oplus\Hc_{-}$, 
$V=\begin{pmatrix} \hfill 1 & \hfill 0 \\ \hfill 0 & \hfill -1\end{pmatrix}$ 
with respect to 
this decomposition of $\Hc$, 
and $\widetilde{J}\colon\Hc\to\Hc$ is an anti-conjugation
such that $\widetilde{J}(\Hc_{\pm})\subseteq\Hc_{\pm}$.
\end{itemize}
As a by-product of 
the classification of the $L^*$-algebras 
(see for instance Theorems 7.18~and~7.19 in \cite{Be06}), 
every (real or complex) topologically simple $L^*$-algebra 
is isomorphic to one of
the classical Banach-Lie algebras associated with 
the Hilbert-Schmidt ideal $\Sg_2(\Hc)$. 

We refer to \cite{dlH72} and \cite{Ne02b} 
for information on the homotopy groups of the classical Banach-Lie groups 
associated with the Schatten ideals. 
The corresponding description 
of homotopy groups actually holds true 
for the classical Banach-Lie groups associated with 
any symmetric norming function $\Phi$. 
It is is clear that \emph{the connected $\1$-component 
of any classical Banach-Lie group associated with $\Phi$ is 
a $\Phi$-reductive linear Banach-Lie group}.  
\qed
\end{example}

It would be interesting to understand 
how the classical Lie algebras 
associated with an operator ideal 
fit in the framework of $\Phi$-reductive Lie algebras. 
Here is a proposition in this connection. 
Recall that a subset 
$A\subseteq\Bc(\Hc)$ is \emph{irreducible} 
if $\{0\}$ and $\Hc$ are the only closed linear subspaces 
of $\Hc$ which are invariant under all operators in $A$. 

\begin{proposition}\label{class}
Let $\Phi$ be a symmetric norming function,   
$\gg$ a $\Phi$-reductive Lie algebra, 
and $\gg_2$ the $L^*$-algebra associated with $\gg$. 
Then the Lie algebra $\gg$ is irreducible 
if and only if $\gg_2$ is irreducible.  
In addition, if $\Sg_\Phi\ne\Sg_1$, 
then $\gg$ is  one of 
the classical (real or complex) Banach-Lie algebras associated with 
the norm ideal $\Sg_\Phi$ 
if and only if $\gg_2$ is  one of 
the classical (real or complex) Banach-Lie algebras associated with 
the Hilbert-Schmidt ideal $\Sg_2$.  
If this is the case, then 
 $\gg$ and $\gg_2$ are classical Lie algebras of the same type.  
\end{proposition}

\begin{proof}
Both algebras $\gg$ and $\gg_2$ are closed under taking the adjoints, 
hence any of them is irreducible if and only if every operator 
that commutes with that algebra is a scalar multiple of 
the identity operator on $\Hc$. 
Therefore it will be enough to prove the following assertion: 
\begin{equation}\label{*}
(\forall T\in\Bc(\Hc))\qquad 
[T,\gg]=\{0\}\iff[T,\gg_2]=\{0\}. 
\end{equation}
In fact, if $[T,\gg]=\{0\}$ then in particular 
$[T,\gg\cap\Fg]=\{0\}$. 
Since the set of finite-rank operators $\gg\cap\Fg$ is dense in $\gg_2$, 
it follows that $[T,\gg_2]=\{0\}$. 
Conversely, if the operator $T$ satisfies 
the latter condition, then we have in particular $[T,\gg\cap\Fg]=\{0\}$. 
By property~(ii) of a $\Phi$-reductive Lie algebra, 
it then follows that $[T,\gg]=\{0\}$.

Now assume $\Sg_\Phi\ne\Sg_1$ and let $\widetilde{\gg}$ be 
the closure of $\gg_2\cap\Fg$ 
with respect to the norm $\Vert\cdot\Vert_\Phi$.
Since $\gg\cap\Fg\subseteq\gg_2\cap\Fg$ 
and $\gg\cap\Fg$ is dense in $\gg$, it then follows that 
$\gg\subseteq\widetilde{\gg}$.  

We are going to prove that $\gg$ is actually a closed ideal 
of the Lie algebra $\widetilde{\gg}$. 
For this purpose it will be enough to show 
that $[\gg\cap\Fg,\gg_2\cap\Fg]\subseteq\gg\cap\Fg$. 
In fact, let $X\in\gg\cap\Fg$ and $Y\in\gg_2\cap\Fg$  arbitrary. 
According to the definition of $\gg_2$, 
there exists a sequence $\{X_j\}_{j\ge1}$ 
in $\gg\cap\Fg$ such that $\Vert X_j-Y\Vert_2\to0$ as $j\to\infty$. 
Then $\Vert [X,X_j]-[X,Y]\Vert_2\to0$ as $j\to\infty$. 
On the other hand, the commutators $\{[X,X_j]\}_{j\ge1}$ 
have the ranks uniformly bounded by $\rank X$, 
hence Lemma~4.3 in \cite{BR05} implies that 
in the topology of $\Vert\cdot\Vert_\Phi$ 
we also have $[X,X_j]\to[X,Y]$ as $j\to\infty$, 
whence $[X,Y]\in\gg\cap\Fg$. 

Now it is easy to see that the assertion holds. 
In fact, if $\gg_2$ is one of 
the classical (real or complex) Banach-Lie algebras associated with 
the Hilbert-Schmidt ideal $\Sg_2$, 
then $\widetilde{\gg}$ is the similar classical Banach-Lie algebra 
associated with the norm ideal $\Sg_\Phi$. 
Since $\Sg_\Phi\ne\Sg_1$, 
it follows as in Proposition~8 on page~92 in \cite{dlH72} 
that $\widetilde{\gg}$ has no non-trivial closed ideals, 
whence $\gg=\widetilde{\gg}$. 
Conversely, 
if $\gg$ is  one of 
the classical (real or complex) Banach-Lie algebras associated with 
the norm ideal $\Sg_\Phi$, then it is obvious 
that the $L^*$-algebra associated with $\gg$ is one of 
the classical $L^*$-algebras. 
\end{proof}

\noindent\textbf{Group decompositions.}
The following theorem supplies a Cartan decomposition 
for some $\Phi$-reductive Banach-Lie groups. 
In this connection, we refer to \cite{Ne02c} 
for a systematic investigation of polar decompositions 
in infinite dimensions
(that is, Cartan decompositions for linear groups). 

\begin{theorem}\label{cartan}
Let $\Phi$ be a symmetric norming function and   
$G$ a $\Phi$-reductive Banach-Lie group 
whose Lie algebra $\gg\subseteq\Sg_\Phi(\Hc)$ 
is one of the classical Lie algebras associated with $\Phi$. 
Denote $\kg=\{X\in\gg\mid X^*=-X\}$, 
$\pg=\{X\in\gg\mid X^*=X\}$, 
and let $K$ be the connected Banach-Lie group 
corresponding to the Lie subalgebra $\kg$ of $\gg$. 
Then $K$ is a Banach-Lie subgroup of $G$ 
and the mapping 
$\varphi\colon K\times\pg\to G$, $(k,X)\mapsto k\exp_GX$, 
is a diffeomorphism.  
In addition, there exists a unique automorphism 
$\Theta\in\Aut(G)$ such that 
$\Lie(\Theta)X=-X^*$ for all $X\in\gg$ and 
$K=\{g\in G\mid\Theta(g)=g\}$. 
\end{theorem}

\begin{proof}
It follows by Proposition~\ref{LL2} that 
we may assume that $G$ is a $\Phi$-reductive \emph{linear} Banach-Lie group. 
Thus let $G\subseteq\1+\Sg_\Phi^{(0)}(\Hc)$ be 
the connected $\1$-component of the classical Banach-Lie group 
associated with the classical Lie algebra $\gg\subseteq\Sg_\Phi(\Hc)$. 
In this case the assertion can be proved by a method 
similar to the one used in Proposition~III.8 of~\cite{Ne02b} 
in the case of the classical groups associated with the Schatten ideals. 
The existence of the automorphism $\Theta\in\Aut(G)$ 
as asserted follows by an application of Proposition~\ref{LL3}. 
\end{proof}

The following theorem states that there exist Iwasawa decompositions 
for the $\Phi$-reductive Banach-Lie groups corresponding 
to the classical Lie algebras, 
provided the symmetric norming function $\Phi$ satisfies 
some reasonable conditions. 

\begin{theorem}\label{iwasawa}
Let $\Phi$ be a mononormalizing symmetric norming function 
whose Boyd indices are nontrivial and   
let $G$ be a $\Phi$-reductive Banach-Lie group 
whose Lie algebra $\gg\subseteq\Sg_\Phi(\Hc)$ 
is one of the classical Lie algebras associated with $\Phi$. 
Denote $\kg=\{X\in\gg\mid X^*=-X\}$, 
$\pg=\{X\in\gg\mid X^*=X\}$, 
and let $K$ be the connected Banach-Lie group 
corresponding to the Lie subalgebra $\kg$ of $\gg$. 
Then there exists $X_0\in\pg$ with the following properties: 
\begin{itemize}
\item[{\rm(a)}] The set $\Ec_{X_0}$ of spectral projections 
of the self-adjoint operator $X_0$ corresponding to the intervals 
of the form $(0,t]$ with $t\in{\mathbb R}$ 
determines a direct sum decomposition 
$\gg=\kg\dotplus\ag_{X_0}\dotplus\n_{X_0} $, 
where $\ag_{X_0}=\{X\in\pg\mid [X,X_0]=0\}$ and 
$\n_{X_0}=\{X\in\gg\cap\Alg(\Ec_{X_0})\mid X(e(\Hc))\ne e(\Hc)
\text{ if }0\ne e\in\Ec_{X_0}\}$.
\item[{\rm(b)}] If we denote by $A$ and $N$ the connected Banach-Lie groups 
corresponding to the closed subalgebras $\ag_{X_0}$ and $\n_{X_0}$ of $\gg$, 
then 
the multiplication map
$\m\colon K\times A\times N\to G$
is a diffeomorphism. 
Moreover,  $A$ and $N$ are simply connected Banach-Lie subgroups of $G$ 
and $AN=NA$. 
\end{itemize}
\end{theorem}

\begin{proof}
Proposition~\ref{lift_iw} allows us to assume that $G$ is 
actually a $\Phi$-reductive \emph{linear} Banach-Lie group. 
The idea of the proof in this case is to start by studying 
the group $G=\GL_\Phi(\Hc)$, which is the largest classical group. 
For this group, the construction of a local Iwasawa decomposition 
of its Lie algebra (assertion~(a)) relies on local spectral theory 
(see \cite{BS01}) along with the properties 
of triangular integrals established in Example~\ref{tri3}. 
As regards the corresponding global Iwasawa decomposition 
(assertion~(b)), one uses Corollary~\ref{fact2}.
See \cite{Be07} for details. 
\end{proof}

\begin{remark}\label{am}
\normalfont
The fundamental groups of the classical Banach-Lie groups 
associated with $\Phi$ 
are always abelian (see \cite{dlH72} and \cite{Ne02b}). 
It then easily follows by Remark~\ref{me6} along with 
Examples \ref{me3} through \ref{me5} that all of the groups 
$K$, $A$, and $N$ that occur in Theorem~\ref{iwasawa} are amenable, 
although the group $G$ itself may not be amenable (see Example~\ref{me5}(a)). 
\qed
\end{remark}

\noindent\textbf{Harish-Chandra decompositions.}
We are going to draw a little closer 
to representation theory, 
which was the main motivation of the present exposition. 
For this purpose we borrow the following definition of 
infinite-dimensional Lie groups of Harish-Chandra type from \cite{NO98}. 
Some good references for such Harish-Chandra decompositions 
in the setting of finite-dimensional Lie groups are 
\cite{Sa80}, \cite{Kn96}, and \cite{Ne00}.  

\begin{definition}\label{HC}
\normalfont
By Banach-Lie group of \emph{Harish-Chandra type} 
we actually mean a 4-tuple $(G,G^{\mathbb C},K,K^{\mathbb C})$ 
consisting of a connected complex Banach-Lie group $G^{\mathbb C}$, 
and and three connected Banach-Lie subgroups $G$, $K$, and $K^{\mathbb C}$ of 
$G^{\mathbb C}$ such that the following conditions are satisfied: 
\begin{itemize} 
\item[{\rm(i)}] 
   the Lie algebra $\gg^{\mathbb C}$ of $G^{\mathbb C}$ is 
the complexification of the Lie algebra $\gg$ of $G$; 
\item[{\rm(ii)}] 
$K^{\mathbb C}$ is a complex Banach-Lie subgroup of $G^{\mathbb C}$  
and the Lie algebra $\kg^{\mathbb C}$ of $K^{\mathbb C}$ is 
the complexification of the Lie algebra $\kg$ of $K$; 
\item[{\rm(iii)}]\, 
there exist connected complex Banach-Lie subgroups $P^{\pm}$ 
of $G^{\mathbb C}$ 
whose Lie algebras $\pg^{\pm}$ have the properties 
$(\ad\,\pg^{\pm})^n\gg^{\mathbb C}=\{0\}$ 
for some integer $n\ge1$ and $\pg^{\pm}\cap\zg(\gg)=\{0\}$, 
where $\zg(\gg^{\mathbb C})$ denotes the center of $\gg^{\mathbb C}$, 
and moreover:  
\begin{itemize}
\item[{\rm(HC1)}]\qquad
we have the direct sum decomposition 
$\gg^{\mathbb C}=\pg^{+}\dotplus\kg^{\mathbb C}\dotplus\pg^{-}$, 
and in addition $[\kg^{\mathbb C},\pg^{\pm}]\subseteq\pg^{\pm}$ and 
$\overline{\pg^{-}}=\pg^{+}$, 
where $X\mapsto\overline{X}$, $\gg^{\mathbb C}\to\gg^{\mathbb C}$, 
is the antilinear involutive map whose fixed-point set is $\gg$;  
\item[{\rm(HC2)}] \qquad
the multiplication mapping 
$P^{+}\times K^{\mathbb C}\times P^{-}\to G^{\mathbb C}$, 
$(p_{+},k,p_{-})\mapsto p_{+}kp_{-}$, 
is a biholomorphic diffeomorphism onto its open image; 
\item[{\rm(HC3)}] \qquad
we have 
$G\subseteq P^{+}K^{\mathbb C}P^{-}$ and 
$G\cap K^{\mathbb C}P^{-}=K$. 
\end{itemize} 
\end{itemize} 
If the groups $G^{\mathbb C}$, $K$, and $K^{\mathbb C}$ 
are singled out in a certain context, 
then we may say that the group $G$ itself is a 
Banach-Lie group of Harish-Chandra type.  
\qed
\end{definition}

\begin{example}\label{SU}
\normalfont
Let $\Phi$ be a symmetric norming function and   
$\Hc=\Hc_{+}\oplus\Hc_{-}$. 
Then the corresponding 
classical real Banach-Lie group of type~(AIII) 
---in the terminology of Example~\ref{classical}--- 
provides an example of group of Harish-Chandra type. 
Specifically, we mean the 4-tuple $(G,G^{\mathbb C},K,K^{\mathbb C})$,   
where 
$$\begin{aligned}
G&=\U_\Phi(\Hc_{+},\Hc_{-}), 
\quad K=\{k\in\U_\Phi(\Hc_{+},\Hc_{-})\mid k(\Hc_{\pm})\subseteq\Hc_{\pm}\}\\
G^{\mathbb C}&=\GL_\Phi(\Hc), \qquad\,
K^{\mathbb C}=\{g\in\GL_\Phi(\Hc)\mid g(\Hc_{\pm})\subseteq\Hc_{\pm}\}.
\end{aligned}
$$
Then the conditions of Definition~\ref{HC} are satisfied 
with the connected complex Banach-Lie subgroups 
$P^{\pm}=\{g\in\GL_\Phi(\Hc)\mid(g-\1)\Hc_{\mp}\subseteq\Hc_{\pm}\}$. 
If we write the operators on $\Hc$ as $2\times 2$ block matrices 
with respect to the orthogonal decomposition $\Hc=\Hc_{+}\oplus\Hc_{-}$, 
then 
$$G\subseteq P^{+}K^{\mathbb C}P^{-}
=\Bigl\{g=\begin{pmatrix} A & B \\ C & D\end{pmatrix}\in\GL_\Phi(\Hc)\mid 
D\in\GL(\Hc_{-})\Bigr\} 
$$
and every element in this set can be factorized as 
$$\begin{pmatrix} A & B \\ C & D\end{pmatrix}
=
\begin{pmatrix} \1 & & BD^{-1} \\ 0 & &  \1\end{pmatrix}
\begin{pmatrix} A-BD^{-1}C & & 0 \\ 0 & & D\end{pmatrix}
\begin{pmatrix} \1 & & 0 \\ D^{-1}C & & \1\end{pmatrix}, $$
where the factors in the right-hand side belong 
to $P^{+}$, $K^{\mathbb C}$, and $P^{-}$, respectively. 
We refer to \cite{NO98} for more details as well as for 
similar examples provided by groups of type (BII)~and~(CI) 
(again in the terminology of Example~\ref{classical} above).
\qed
\end{example}

\begin{remark}\label{HC_more}
\normalfont
As noted in \cite{NO98} (see also \cite{Ne00}), 
the objects involved in the definition of a group of Harish-Chandra type 
(Definition~\ref{HC} above) have the following additional properties: 
\begin{itemize}
\item[{\rm(a)}] 
 If we denote $\Lie(P^{\pm})=\pg^{\pm}$, 
 then the exponential maps $\exp_{P^{\pm}}\colon\pg^{\pm}\to P^{\pm}$ 
 are biholomorphic diffeomorphisms and 
  the complex Banach-Lie groups $P^{\pm}$ are nilpotent and simply connected. 
In particular there exist the logarithm maps 
$\log_{P^{\pm}}=(\exp_{P^{\pm}})^{-1}\colon P^{\pm}\to\pg^{\pm}$. 
\item[{\rm(b)}] 
 There exists an open connected $K$-invariant subset $\Omega\subseteq\pg^{+}$ 
 such that the mapping 
 $\Omega\times K^{\mathbb C}P^{-}\to GK^{\mathbb C}P^{-}$, 
 $(Z,p)\mapsto(\exp_{G^{\mathbb C}}Z)p$, 
 is a well-defined biholomorphic diffeomorphism. 
\item[{\rm(c)}]
 Let $\zeta^{\pm}\colon P^{+}K^{\mathbb C}P^{-}\to P^{\pm}$ 
and $\kappa\colon P^{+}K^{\mathbb C}P^{-}\to K^{\mathbb C}$ 
be the natural projections (see condition~(HC2) in Definition~\ref{HC}) 
and 
$\Xi=\{(g,Z)\in G^{\mathbb C}\times\pg^{+}\mid 
g\exp_{G^{\mathbb C}}Z\in P^{+}K^{\mathbb C}P^{-}\}$, 
which is an open neighborhood of $G\times\Omega$ in 
 $G^{\mathbb C}\times\pg^{+}$. 
Define 
$\Xi\to\pg^{+}$, 
$(g,Z)\mapsto g.Z=\log_{P^{+}}(\zeta^{+}(g\exp_{G^{\mathbb C}}Z))$ 
and 
$J\colon\Xi\to K$, 
$(g,Z)\mapsto J(g,Z)=\kappa(g\exp_{G^{\mathbb C}}Z)$.
Then  the mapping 
$(g,Z)\mapsto g.Z$ 
defines a transitive action of $G$ upon $\Omega$ by 
biholomorphic diffeomorphisms 
and $J$ has a cocycle property with respect to this action. 
\item[{\rm(d)}]
We have $0\in\Omega$ and the isotropy group of $0$ is equal to $K$. 
Thus the aforementioned action leads to a 
$G$-equivariant diffeomorphism $G/K\simeq\Omega$. 
\end{itemize}
These are some of the ideas that allow one 
to define as in \cite{NO98} 
natural reproducing kernels on $\Omega$ 
out of certain representations of the Banach-Lie group $K$.  
In this way one ends up with 
the corresponding reproducing kernel Hilbert spaces of holomorphic functions 
on $\Omega$ which carry representations of the bigger Banach-Lie group $G$ 
of Harish-Chandra type. 
\qed
\end{remark}

\begin{remark}\label{final}
\normalfont
Applications of a different type of reproducing kernels 
in representation theory of Banach-Lie groups 
can be found in \cite{BR06} and \cite{BG07}. 
The viewpoint held in the first of these papers is 
somehow dual to the one of Remark~\ref{HC_more} 
(in the sense of duality theory of symmetric spaces), 
i.e., the bases of the corresponding 
vector bundles are homogeneous spaces of compact type. 
The complexification of this picture of compact type 
is analyzed in \cite{BG07} along with the relationship 
to Stinespring dilation theory, which eventually leads to 
geometric models for representations of operator algebras. 
\qed
\end{remark}

\noindent\textbf{Acknowledgments.} 
We wish to thank Professor Jos\'e Gal\'e for drawing our attention to 
some pertinent references, 
to 
Professor Hendrik Grundling for kindly sending us the manuscript~\cite{Gr}, 
and to Professor Gary Weiss for comments that helped us to improve 
the exposition. 
Partial support from the grant GR202/2006 
(CNCSIS code 813) is acknowledged.

\printindex
\end{document}